\documentclass[12pt]{amsart}
\usepackage[scale=0.86]{geometry}
\usepackage{geometry} 
\usepackage{graphicx}
\usepackage{bbm,enumerate}
\usepackage{mathrsfs}
\usepackage{color}

\makeatletter
\@namedef{subjclassname@2020}{%
  \textup{2020} Mathematics Subject Classification}
\makeatother

\theoremstyle{definition}
\newtheorem{definition}{Definition}%Extra square-bracket argument achives that the numbering is the same as for definition (single uniform counter). 
\theoremstyle{theorem}
\newtheorem{proposition}[definition]{Proposition}
\newtheorem{lemma}[definition]{Lemma}
\newtheorem{theorem}[definition]{Theorem}

\numberwithin{equation}{section}
\numberwithin{definition}{section}
\theoremstyle{remark}
\newtheorem{remark}[definition]{Remark}

\newtheorem{example}[definition]{Example}

\allowdisplaybreaks
\title[Extinction and explosion in logistic CSBPs ]{Local explosions and extinction in continuous-state branching processes with  logistic competition}
\author{Cl\'ement Foucart}\thanks{foucart@math.univ-paris13.fr, Universit\'e Paris 13,  Laboratoire Analyse, G\'eom\'etrie \& Applications UMR 7539 Institut Galil\'ee and Ecole Polytechnique, Centre de Math\'ematiques Appliqu\'ees, UMR 7641}
\usepackage[pagebackref,colorlinks,citecolor=Plum,urlcolor=Periwinkle,linkcolor=DarkOrchid]{hyperref}

\usepackage[dvipsnames]{xcolor}

\newcommand{\zmin}{\ensuremath{Z^{\mathrm{min}}}}

\newcommand{\ddr}{\mathrm{d}}

\date{30/09/2025}
\begin{document}

\begin{abstract}
We study by duality methods the extinction and explosion times of continuous-state branching processes with logistic competition (LCSBPs) and identify the local time at $\infty$ of the process when its boundary $\infty$ is regular reflecting. The main idea is to introduce a certain ``bidual" process $V$ of the LCSBP $Z$.  The latter is the Siegmund dual process of the process $U$, that was introduced in \cite{MR3940763}  as the Laplace dual of $Z$. By using both dualities, we shall relate local explosions and the extinction of $Z$ to local extinctions and the explosion of the process $V$. The process $V$ being a one-dimensional diffusion on $[0,\infty]$, many results on diffusions can be used and transfered to $Z$. A concise study of Siegmund duality for one-dimensional diffusions and their boundaries is also provided.
\end{abstract}

\maketitle

\section{Introduction}
Continuous-state branching processes (CSBPs) with logistic competition  are Markov processes that have been introduced by Lambert \cite{MR2134113} to model the size of a stochastic population in which a self-regulation dynamics is taken into account. These Markov processes are valued in $[0,\infty]$, the one-point compactification of the half-line, and can be seen as classical branching processes on which a deterministic competition pressure between pair of individuals,  parametrized by a real value $c>0$, is superimposed. For instance, if the branching dynamics are given by a critical Feller diffusion, the logistic CSBP (LCSBP) is solution to the stochastic differential equation (SDE):
\[\ddr Z_t=\sigma\sqrt{Z_t}\ddr B_t-\frac{c}{2}Z_t^{2}\ddr t,\quad Z_0=z\in (0,\infty),\]
with $(B_t,t\geq 0)$ a Brownian motion and $\sigma>0$.  In the general setting, the diffusive component above is replaced by the full dynamics of a CSBP (see e.g. Li~\cite[Chapter~9]{Li-book} and Kyprianou~\cite[Chapter~12]{Kyprianoubook}), which may exhibit positive jumps of arbitrarily large size. Those dynamics are determined by a L\'evy--Khintchine function $\Psi$ on $\mathbb{R}_+$, known as the \emph{branching mechanism}. 

Processes with competition do not satisfy any natural branching or affine properties, which makes their study challenging; however  it has been observed~\cite{MR3940763} that a logistic CSBP $Z$ lies in duality with a certain $[0,\infty]$-valued diffusion process $U$, referred to as Laplace dual of $Z$: namely for any $z\in [0,\infty]$, $x\in (0,\infty)$  and $t\geq 0$, 
\[\mathbb{E}_z[e^{-xZ_t}]=\mathbb{E}_x[e^{-zU_t}].\]
The notations $\mathbb{P}_z$ and $\mathbb{E}_z$ stand for the law of the underlying process started from $z$ and its expectation. We will not address here pathwise relationships, and we keep this notation for all processes,  which can be thought of as defined on different probability spaces. 

Duality relations are well-known in the theory of interacting particle systems, see e.g. Giardin{\`a} and Redig~\cite{dualityliebook} and Hutzenthaler and Wakolbinger~\cite{HutzenthalerWakolbinger} for a spatial setting close to ours. 

It has been established in \cite{MR3940763} that the boundary $\infty$ is accessible for certain LCSBPs. In other words some populations with  very strong reproduction  can escape from self-regulation and explode despite the quadratic competition.  Duality relationships map the \textit{entrance} laws of one process to the \textit{exit} laws of the other, Cox and R\"osler~\cite{MR724061}. Behaviors of $Z$ at its boundaries $0$ and $\infty$ are thus intrinsically related to those of the diffusion $U$ at $\infty$ and $0$ respectively.   A logistic CSBP can have actually its boundary $\infty$ as an exit (it hits $\infty$ and stays there), as a  regular reflecting boundary (the process returns to it at a set of times of zero Lebesgue measure), or as instantaneous entrance (the process leaves immediately the boundary and never visits it again).
 
The aim of this article is to push further the analysis of the process $Z$ by studying  the laws of three key quantities: the extinction time, the first explosion time and, finally, the local time at $\infty$ when $\infty$ is a regular reflecting boundary. To achieve this, we make use of a second duality relationship and introduce the Siegmund dual process of $U$: namely the process $V$ satisfying for any $x,y\in (0,\infty)$ and $t\geq 0$,
\[\mathbb{P}_x(U_t<y)=\mathbb{P}_y(x<V_t).\]
A general theorem due to Siegmund \cite{MR0431386}, recalled in Section \ref{diffusionsec}, ensures the existence of the process $V$. We summarize both dualities in the following diagram: \begin{equation}\label{dualitydiagram} Z \overset{\textbf{Laplace dual}}{\longleftrightarrow} U \overset{\textbf{Siegmund dual}}{\longleftrightarrow} V.
\end{equation} 
The process $V$, referred to as the \textit{bidual} process of $Z$, turns out to be a $[0,\infty]$-valued diffusion and plays a central role in our approach. 

One of the main contributions of this work is to demonstrate how the study of $Z$ can be effectively carried out with the aid of the bidual process. In a way, we shall see how both dualities involved serve to map \textit{entrance} laws of $Z$ to \textit{entrance} laws of $V$. Indeed, combining these two dualities one shall check the following relationship between $Z$ and $V$: 
\begin{equation}\label{joiningduals}\mathbb{E}_z(e^{-xZ_t})=\int_{0}^{\infty}ze^{-zy}\mathbb{P}_{y}(V_t>x)\ddr y, \text{ for } x,z \in (0,\infty),\ t\geq 0.
\end{equation}
When $z$ tends to $\infty$, the identity \eqref{joiningduals} yields the following link between the entrance laws of $Z$ and $V$: 
\begin{equation}\label{joiningdualsatboundary}
\mathbb{E}_\infty(e^{-xZ_t})=\mathbb{P}_0(V_t>x) \text{ for } x,t\geq 0.
\end{equation} 
Those identities will be established in the forthcoming Section \ref{section:keyidentities}. We shall see how  \eqref{joiningduals} propagates to  the laws of the times of extinction and first explosion of $Z$, Theorem \ref{theorem1} and Theorem \ref{theorem2}, and to the local time, Theorem \ref{theorem3}. More precisely, we establish that the latter has the same law as the local time of $V$ at $0$. This, in turn, enables us to compute the Hausdorff dimension of the set of explosion times and to relate the excursion measures of $Z$ and $V$, respectively at $\infty$ and $0$, through their resolvents; Theorems~\ref{Hausdorff} and~\ref{theoremdualinsiden}. We will finally determine the law, under the excursion measure, of the infimum of $Z$, see Theorem~\ref{infimuminexcursion}. 

As a tool for our analysis, we shall also study Siegmund duality of one-dimensional diffusions on $[0,\infty]$ in Section~\ref{diffusionsec}. Our approach will rely on arguments different from those in~\cite{MR724061} and will complement the results presented there.

The paper is organised as follows. In Section \ref{background}, we recall fundamental elements on one-dimensional diffusion processes and the classification of their boundaries. We then provide some background on logistic CSBPs, how they can be constructed up to hitting their boundaries and how extended processes are dual to certain generalized Feller diffusions. Our main results are stated in Section \ref{mainresults}. Section \ref{remarkwithoutcompetition} sheds some light on the case without competition. The proofs are  provided in Section \ref{proofsec} and will make use of some general results on Siegmund duality established independently in Section \ref{diffusionsec}.

\section{Preliminaries}\label{background}
\noindent\textbf{Notations.}  We denote by $C_c^2(0,\infty)$ (resp.\ $C_c^\infty(0,\infty)$) the space of twice (resp.\ infinitely) continuously differentiable functions with compact support in $(0,\infty)$. Similarly, $C^1(0,\infty)$ (resp.\ $C^2(0,\infty)$) denotes the space of (resp.\ twice) continuously differentiable functions on $(0,\infty)$. The space of bounded Borelian functions on $[0,\infty]$ is denoted by $B_b([0,\infty])$. We highlight that in all the article we take the conventions $\infty \times 0=0$ and $0 \times \infty=\infty$. 

\subsection{Terminology}\label{terminology}
Given a c\`adl\`ag strong Markov process taking values in $[0,\infty]$, we say that the boundary $0$ (respectively $\infty$) is accessible if, when the process starts from $(0,\infty)$, it hits $0$ (respectively $\infty$) with positive probability. Otherwise, we say that the boundary is inaccessible. 

When a boundary is inaccessible, it can be either  an entrance or a natural boundary. In the entrance case, although it cannot hit the boundary, the process can be started from it, that is to say, if the process is initially at the boundary, then it will leave it at some future time. In the natural case, the process neither can leave nor hit the boundary.  

When a boundary is accessible, it can be either an exit or a regular boundary. In the exit case, the process cannot leave the boundary and thus stays at it after it has reached it. In the regular case, the process can leave the boundary (in various ways) if it is not stopped upon reaching it. We shall distinguish two cases for a regular boundary. The boundary will be called regular reflecting when the time spent by the process at the boundary has a zero Lebesgue measure. When the process is stopped at a regular boundary, the boundary is said to be regular absorbing. We stress that in the article, all processes under consideration will leave instantaneously a boundary that is non-absorbing (regular reflecting or entrance boundary).

Lastly, a regular boundary is regular for itself if the process returns  immediately after having left it. For a broad class of processes, this entails the existence of a local time at the boundary and the decomposition of the paths into excursions away from the boundary along a Poisson point process. We refer the reader for instance to Blumenthal and Getoor's book \cite[Chapter V, Section 3]{zbMATH03272022} and Bertoin \cite[Chapter 4, Section 2]{Bertoin96} for two different constructions of the local time.

\subsection{One-dimensional diffusions on $[0,\infty]$}
Throughout this section, we consider two continuous functions on $[0,\infty)$,  $\sigma$ and $\mu$, that are  locally Lipschitz on $(0,\infty)$. We also assume that $\sigma$ is strictly positive on $(0,\infty)$.
\subsubsection{Stochastic differential equations and martingale problem}\label{sec:MPdiffusion}
 Consider the following SDE
\begin{equation}\label{sde0} \ddr U_t=\sigma(U_t)\ddr B_t+\mu(U_t)\ddr t, \ U_0=x\in (0,\infty),
\end{equation}
for some Brownian motion $(B_t,t\geq 0)$. Then, there exists a unique weak solution $(U_t,t\geq 0)$, defined up to the stopping time $S:=\inf\{t>0: U_t\notin(0,\infty)\}$, see for instance Revuz and Yor's book \cite[Exercice 2.10, page 383]{MR1725357}. 
We call \textit{minimal} solution the process obtained by extending $(U_t,t\geq 0)$ beyond $S$ via $U_t=U_{S}$ for all $t\geq S$. It has a version with continuous sample paths and for any $t\geq S$, $U_t=0$ if $S=\tau_0:=\inf\{t>0: U_t=0\}<\infty$ and $U_t=\infty$ if $S=\tau_\infty:=\inf\{t>0: U_t=\infty\}<\infty$.  In other words, we stop the process at its first hitting time of the boundary $0$ and $\infty$.

Moreover, a process $(U_t,t\geq 0)$ is the minimal (weak) solution of Equation \eqref{sde0} if and only if it has absorbing boundaries and it satisfies  the following martingale problem $\mathrm{(MP)}_U$: for any $f\in C^2_c(0,\infty)$, the process
\begin{equation}\label{MPU} \left(f(U_{t})-\int_{0}^{t}\mathscr{A}f(U_{s})\ddr s, t\geq 0\right)\text{ is a martingale,}
\end{equation}
where $\mathscr{A}$ is called the generator and takes the form
\[\mathscr{A}f(x)=\frac{1}{2}\sigma^2(x)f''(x)+\mu(x)f'(x),\ f\in C^2(0,\infty), x\in (0,\infty).\]
We refer e.g. to  Durrett's book \cite[Section 6.1]{MR1398879} for a study of $\mathrm{(MP)}_U$. 

The minimal solution does not take into account the behaviors at the boundaries, in the sense that as soon as one boundary is accessible, it is absorbing for the minimal diffusion. Other solutions  to the martingale problem $\mathrm{(MP)}_U$ besides the minimal one may exist, and all the different behaviors described in Section \ref{terminology} can possibly happen at a non-natural boundary for some functions $\sigma$, $\mu$.

In order to classify them, we need the scale function and speed measure.  

\subsubsection{Scale function and speed measure}\label{backgroundspeedscale}
Let $x_0$, $y_0$ be arbitrary fixed  points in $(0,\infty)$. Set $s_U(x):=\exp\left(-\int_{x_0}^x\frac{2\mu(u)}{\sigma^2(u)}\ddr u\right)$ for $x\in (0,\infty)$ and 
\begin{equation}\label{scalefunction}
S_U(y):=\int_{y_0}^{y}s_U(x)\ddr x=\int_{y_0}^{y}\exp\left(-\int_{x_0}^x\frac{2\mu(u)}{\sigma^2(u)}\ddr u\right)\ddr x, \ y\in (0,\infty).
\end{equation}
We call $S_U$ the scale function\footnote{they are defined up to affine transformations.}  and shall also denote by $S_U$ the associated Stieltjes measure: $S_U(\ddr x)=s_U(x)\ddr x$. Let $m_U(x):=\frac{1}{\sigma^2(x)s_U(x)}$ for any $x\in (0,\infty)$ and \begin{equation}\label{speedmeasure}
M_U(y):=\int_{y_0}^{y}m_U(x)\ddr x=\int_{y_0}^{y}\frac{1}{\sigma^2(x)}\exp\left(\int_{x_0}^x\frac{2\mu(u)}{\sigma^2(u)}\ddr u\right)\ddr x, \ y\in (0,\infty).
\end{equation} 
We denote also  by $M_U$ the associated Stieltjes measure, this is the so-called speed measure, $m_U$ being the speed density measure. An important fact is that the one-dimensional law of the diffusion $U$ admits a density with respect to the speed measure $M_U$, Rogers and Williams \cite[Theorem 50.11, Chapter V]{zbMATH01515832}. In our case the latter will always be absolutely continuous and in particular for any $t>0$, the law of $U_t$ has no atom in $(0,\infty)$.  

\subsubsection{Feller's conditions}\label{secFellercond}
The classification of boundaries of one-dimensional diffusions is completely understood. We refer for instance to Karlin and Taylor's book \cite[Chapter 15, Section 6]{zbMATH03736679}. We briefly recall the integral tests that are specifying the behavior of a diffusion at its boundaries. 

For any $l\in [0,\infty]$, define the integral tests $I_U$ and $J_U$ by
\begin{equation}\label{IJ} I_U(l):=\int_{l}^xS_U(l,x]\ddr M_U(x) \text{ and } J_U(l):=\int_{l}^xS_U[u,x]\ddr M_U(u).
\end{equation}
The point $l\in [0,\infty]$ is accessible if and only if $I_U(l)<\infty$. The integral test $J_U(l)$ specifies whether the process can leave the boundary $l$ or not.  The following analytical classification of boundaries can be found for instance in \cite[Table 6.2, page 234]{zbMATH03736679}. 

\begin{table}[htpb]
\begin{tabular}{|c|c|}
\hline
Feller's conditions & Boundary of $U$ \\
\hline
$S_U(0,x]<\infty \text{ and } M_U(0,x]<\infty$ &  $0$ regular \\
\hline
$S_U(0,x]=\infty \text{ and } J_U(0)<\infty$ &  $0$ entrance \\
\hline
$M_U(0,x]=\infty \text{ and } I_U(0)<\infty$ &  $0$ exit \\
\hline
$I_U(0)=\infty \text{ and } J_U(0)=\infty$ &  $0$ natural \\
\hline
\end{tabular}
\vspace*{4mm}
\caption{Boundaries of $U$.}
\label{Fellerconditions}
\vspace*{-5mm}
\end{table}
\noindent By replacing everywhere $0$ by $\infty$ in Table \ref{Fellerconditions}, we get the classification for the boundary $\infty$. 

\medskip

In the regular case, an extra information on the behavior at the boundary is needed to completely understand the process, see e.g. Borodin and Salminen \cite{MR1912205} for the complete classification. We will only consider the two extreme possibilities namely reflection and absorption (sticky behavior interpolates between the two and is not considered here).

\medskip

When a boundary is natural or is absorbing (i.e. exit or regular absorbing), we say that the boundary is attracting if the process has a positive probability to converge towards it. We have the following classification \cite[Proposition 5.22, page 345]{karatzas}:

\begin{table}[htpb]
\begin{tabular}{|c|c|}
\hline
Conditions & Boundary of $U$ \\
\hline
$S_U(0,x]<\infty$ &  $0$ attracting \\
\hline
$S_U[x,\infty)<\infty$ &  $\infty$ attracting \\
\hline
\end{tabular}
\vspace*{4mm}
\caption{Attracting boundaries of $U$.}
\label{Attractingconditions}
\end{table}
%\vspace*{-3mm}

\noindent Moreover for all $x\in (0,\infty)$,
\[\mathbb{P}_x(U_t\underset{t\rightarrow \infty}{\longrightarrow} 0)=1-\mathbb{P}_x(U_t\underset{t\rightarrow \infty}{\longrightarrow} \infty)=\frac{S_U[x,\infty)}{S_U(0,\infty)}.\]  
In particular, when both boundaries are attracting, the process will converge towards one of them almost surely.
\subsubsection{Feller's construction of diffusions with $0$ regular reflecting boundary}\label{sec:fellerconstruction}
The diffusion  solution to \eqref{sde0} with, say, the boundary $0$ regular -- in the sense that it is accessible non-absorbing --  can be constructed from a time-changed reflected Brownian motion. We explain this briefly here and refer  to Karatzas and Shreve \cite[Section 5.5-B, pages 339-340]{karatzas} and  Durrett \cite[Section~6.5]{MR1398879} for details on the following construction. 
 
Let $U^{\mathrm{a}}$ be the diffusion of $\eqref{sde0}$ absorbed at $0$ with given coefficients $\sigma$ and $\mu$. Call the associated scale function $S$ and the speed measure $M$. Assume that the boundary $0$ is regular in the sense of Table \ref{Fellerconditions}, i.e. $S(0,x]<\infty$ and $M(0,x]<\infty$ for some $x>0$.

We construct a process that solves  $(\mathrm{MP})_{U}$ -- and  constitutes a weak solution to \eqref{sde0} -- with its boundary $0$ being regular non-absorbing. First we transfer the problem in natural scale, namely we ``remove" the drift, see e.g. \cite[Section 6.5, page 229]{MR1398879}, with the help of the scale function. Choose the scale function $S$ such that $S(0)=0$. The diffusion $(S(U_t^{\mathrm{a}}),t\geq 0)$ is in natural scale, i.e. its scale function is the identity, and has speed density measure $1/h$, defined by 
\begin{equation}\label{densityspeedmeasure}
h(y):=S'(S^{-1}(y))^2 S^{-1}(y) \text{ for } y \in [0,\infty),
\end{equation}
see \cite[Equation (1.5), Section 6.1, page 212]{MR1398879}. Then, extend $h$ on $\mathbb{R}$ by $h(-y)=h(y)$ for all $y$, let $(X_t,t\geq 0)$ be the $\mathbb{R}$-valued diffusion in natural scale with speed density measure $m_X(y):=1/h(|y|)$ for all $y\in \mathbb{R}$, and finally define \begin{equation}\label{extended-diff}
U_t:=S^{-1}(|X_t|) \text{ for all } t\geq 0. 
\end{equation}
In order to check that $U$ indeed has $0$ non-absorbing it suffices to verify that for some $x>0$, the point $-x$ is accessible for $X$ started from $0$. A simple calculation yields
\[S_X(-x,0]=x<\infty \text{ and }M_X(-x,0]=\int_{0}^{x}m_X(y)\ddr y=M(0,{S^{-1}}(x)]<\infty.\]
This ensures that $I_X(-x)<\infty$, see \eqref{IJ}, $-x$ is thus accessible (it is actually regular), hence the process $U$ leaves its boundary~$0$. The fact that $0$ is regular for itself and reflecting for $U$ can  also be checked (for the latter notice that $M_X$ has no atom). 
\subsection{Martingale problem of LCSBPs and construction of the minimal LCSBP}\label{lamperticonstruction}
 
\subsubsection{Generator of LCSBPs}
Let $\Psi$ be a branching mechanism, namely a function of the Lévy-Khintchine form : 

\begin{equation}
\label{LK}\Psi(x)= -\lambda+\frac{\sigma^2}{2}x^2+\gamma x+\int_{0}^{\infty}\left(e^{-xh}-1+xh\mathbbm{1}_{\{h\leq 1\}}\right)\pi(\ddr h)\quad
\text{ for all }x\geq 0,
\end{equation}
where $\lambda\geq 0,\sigma\geq 0, \gamma \in \mathbb{R}$ and $\pi$ is a Lévy measure on $(0,\infty)$ such that $\int_0^{\infty} (1\wedge x^2)\pi(\ddr x)<\infty$. 

Denote by $\mathscr{L}^\Psi$ the extended\footnote{in the sense that it produces local martingales.} generator of the CSBP($\Psi$) and let $\mathcal{D}$ be the space of functions $$\mathcal{D}:=\{f\in C^2 (0,\infty): \text{ the limit } f(\infty):=\underset{z\rightarrow \infty}{\lim} f(z) \text{ exists in }\mathbb{R}\}.$$ For any $f\in \mathcal{D}$, $z\in (0,\infty)$
\begin{equation}\label{genCSBP}\mathscr{L}^\Psi f(z):=\frac{\sigma^2}{2} zf''(z) +\gamma zf'(z)+\lambda z(f(\infty)-f(z))+z\int_{0}^{\infty}\left(f(z+h)-f(z)-hf'(z)\mathbbm{1}_{\{h\leq 1\}}\right)\pi(\ddr h),
\end{equation}
see e.g. Silverstein \cite[Page 1045]{MR0226734}. Notice the jump term from $z$ to $\infty$ at rate $\lambda z$ and observe that when $f$ is vanishing at $\infty$, $\lambda z(f(\infty)-f(z))=-\lambda zf(z)$ for all $z\in (0,\infty)$.

The function $\Psi$ governs the reproduction in the population. In order to take into account the competition term, the generator $\mathscr{L}$ of the LCSBP$(\Psi,c)$ is defined as follows: for any $f\in \mathcal{D}$ 
and $z\in (0,\infty)$, 
\begin{equation}\label{genLCSBP}\mathscr{L}f(z):=\mathscr{L}^\Psi f(z)-\frac{c}{2}z^2f'(z).\end{equation}
We define the LCSBPs with parameter $(\Psi,c)$ as Markov processes solutions to the following martingale problem $\mathrm{(MP)}_Z$: For any $f\in C^2_c(0,\infty)$,  the process
\begin{equation}\label{MPZ} \left(f(Z_t)-\int_{0}^{t}\mathscr{L}f(Z_s)\ddr s, t\geq 0\right)\text{ is a martingale.}
\end{equation}
\noindent 
There exists a unique solution of $\text{(MP)}_Z$ stopped when reaching the boundaries $0$ and $\infty$, see \cite[Section 4]{MR3940763}. We shall refer to it as the \textit{minimal} LCSBP$(\Psi,c)$, since the process does not evolve anymore after it has reached the boundaries. We explain briefly a construction below. 
\subsubsection{Minimal LCSBP}\label{sec:Zmindef}
Following Lambert's idea \cite[Definition 3.2]{MR2134113}, a simple construction of the process absorbed when reaching its boundaries, is provided by time-changing in Lamperti's manner a generalized Ornstein-Uhlenbeck process (GOU)  $(R_t,t\geq 0)$ stopped when reaching $0$. This latter process is solution to the stochastic equation
\begin{equation} \label{OU} \ddr R_t=\ddr X_t-\frac{c}{2}R_t\ddr t,\ R_0=z,\quad  \text{for all }t\leq \sigma_0, 
\end{equation}
where $(X_t, t\geq 0)$ is a spectrally positive L\'evy process with Laplace exponent $\Psi$ -- if $\lambda>0$, it jumps to $\infty$ at an independent exponential time $\mathbbm{e}_\lambda$ -- and where  $\sigma_0$ denotes the first passage time below $0$ of $R$ . 
Define the additive functional $\theta$ and its right-inverse $C$ by \begin{equation}\label{timechangedef} t\mapsto \theta_t:=\int_{0}^{t\wedge \sigma_0}\frac{\ddr s}{R_s}\in [0,\infty] \text{ and } \ t\mapsto C_t:=\inf\{u\geq 0:\ \theta_u>t\}\in [0,\infty],
\end{equation}
with the usual convention $\inf\{\emptyset\}=\infty$. The Lamperti time-change of the stopped process $(R_t,t\geq 0)$ is the process $(\zmin_t,t\geq 0)$ defined by 
\begin{align}\label{zmin}
\zmin_t&=
\begin{cases}
R_{C_t} &  0\leq t<\theta_\infty,\\
0& t\geq \theta_\infty \text{ and }\sigma_0<\infty,\\
\infty& t\geq \theta_\infty   \text{ and } \sigma_0=\infty.
\end{cases}
\end{align}
This process is a c\`adl\`ag solution to $\text{(MP)}_Z$, see \cite[Lemma 4.1]{MR3940763}, and is absorbed whenever it reaches $0$ or $\infty$. The process $\zmin$ is not always the only solution of $\mathrm{(MP)}_Z$. We will describe in the next section solutions with the boundary $\infty$ non-absorbing. \\

We will need a framework slightly more general than martingales associated with compactly supported functions.
\begin{lemma}
Let $f\in \mathcal{D}$, the process
\begin{equation}\label{Mzmin}
(M^{\zmin}_t)_{0\leq t<\zeta_0\wedge \zeta_\infty}:=\left(f(Z^{\min}_{t})-\int_0^{t}\mathscr{L}f(Z^{\min}_s)\ddr s\right)_{ 0\leq t<\zeta_0\wedge \zeta_\infty}
\end{equation} is a local martingale. 
\end{lemma}
\begin{proof}
This  follows from the construction of $Z^{\min}$ in \eqref{zmin}. Indeed, from the stochastic equation~\eqref{OU}, we see by an application of Itô's lemma, see e.g. \cite[Theorem 4.57]{JacodShiryaev}, that for any bounded $f\in C^2(0,\infty)$, the process \[(M^{R}_t,t\geq 0):=\left(f(R_{t})-\int_0^{t}\mathscr{L}^Rf(R_s)\ddr s\right)_{0\leq t<\sigma_0\wedge \mathbbm{e}_\lambda}\] is a local martingale, with  $$\mathscr{L}^Rf(z):=\mathrm{L}^{\Psi}f(z)-\frac{c}{2}zf'(z), \ z\in (0,\infty)$$ where $\mathrm{L}^{\Psi}$ denotes the generator of $X$.  Observing then \eqref{genCSBP} and \eqref{genLCSBP}, it follows that  for any $z\in (0,\infty)$,
$$\mathscr{L}f(z)=z\mathscr{L}^Rf(z).$$ By the definition of $(C_t,t>0)$, see \eqref{timechangedef}, $\int_0^{C_t}\frac{\ddr s}{R_{s\wedge \sigma_0}}=t$ for all $t\geq 0$. Consequently, 
\begin{center}
$C_t=\int_0^tZ_s^{\mathrm{min}}\ddr s$ and $\ddr C_t=\zmin_t\ddr t$ for all $t\geq 0$ a.s..
\end{center} Since the map $t\mapsto C_t$ is continuous, $(M^R_{C_t},t\in [0,\theta_\infty))$ is also a local martingale, see  for instance \cite[Proposition 1.5, Chapter V]{MR1725357} and Vidmar \cite[Item (ii), page 1663]{zbMATH07639718}. Moreover by construction $\theta_\infty=\zeta_0\wedge\zeta_\infty$ a.s. and for any $t\geq 0$, we get by the change of variable $u=C_s$,
\begin{align*}
M^R_{C_t}&=f(R_{C_t})-\int_0^{C_t}\mathscr{L}^{R}f(R_u)\ddr u\\
&=f(R_{C_t})-\int_0^{t}\mathscr{L}^{R}f(R_{C_s})\ddr C_s\\
&=f(Z^{\min}_{t})-\int_0^{t}\mathscr{L}f(Z^{\min}_s)\ddr s=M^{\zmin}_t.
\end{align*} 
This entails that $(M^{\zmin}_t, 0\leq t< \zeta_0\wedge \zeta_\infty)$ is a local martingale. 
\end{proof}

\subsection{Boundary behaviors of CSBPs and LCSBPs}\label{backgroundboundary}

\subsubsection{CSBPs}\label{sec:csbp} 
When there is no competition, i.e. $c=0$, the construction in \eqref{zmin} above is known as the Lamperti's transformation for CSBPs. The process $(Z_t^{\min},t\geq 0)$ is in this case a CSBP$(\Psi)$, see e.g. \cite[Theorem 12.2]{Kyprianoubook}. Call it $(Y_t,t\geq 0)$. It is known that the semigroup of $(Y_t,t\geq 0)$ satisfies the identity
\begin{equation}\label{cumulant}
\mathbb{E}_z[e^{-xY_t}]=e^{-zu_t(x)},
\end{equation}
with $(u_t(x),t\geq 0)$ the unique solution to \begin{equation}\label{cumulantu}\frac{\ddr}{\ddr t} u_t(x)=-\Psi(u_t(x)) \text{ with } u_0(x)=x \in (0,\infty).\end{equation}
The map $(u_t(x),t\geq 0)$ cannot hit the boundaries $0$ and $\infty$, see e.g. Silverstein \cite[Pages 1046-1047]{MR0226734}, and therefore the boundaries $\infty$ and $0$ of $(Y_t,t\geq 0)$ are absorbing. We also plainly see from \eqref{cumulant} that if $0$ (respectively $\infty$) is an entrance for $(u_t,t\geq 0)$, i.e. 
\begin{center}
$u_t(0):=\underset{x\rightarrow 0 \atop x>0}{\lim} \downarrow u_t(x)>0$ for $t>0$, (respectively $u_t(\infty):=\underset{x\rightarrow \infty}{\lim} \uparrow u_t(x)<\infty$ for $t>0$),
\end{center} then the CSBP $Y$ will reach $\infty$ (respectively $0$) with positive probability. The conditions for 
$u_t(0)>0$ and $u_t(\infty)<\infty$ are respectively the integral tests \begin{equation}\label{GreyDynkincond}\int_0\frac{\ddr x}{-\Psi(x)}<\infty \text{ (Dynkin's condition)}\text{ and }\int^{\infty}\frac{\ddr x}{\Psi(x)}<\infty \text{ (Grey's condition)},\end{equation}
see e.g. \cite[Theorems 12.3 and 12.5]{Kyprianoubook}. Note that when these integrals are finite, the integrand is always positive near the boundary.
\subsubsection{LCSBPs} When there is competition, i.e. $c>0$, the boundary behaviors are richer. We briefly recall here the results of \cite[Section 3]{MR3940763}. A striking difference between CSBPs and LCSBPs is that, whereas CSBPs cannot restart from the boundary \(\infty\), in most cases where this boundary is accessible the LCSBP can restart continuously from it. 

More rigorously, we call \emph{extension} of $\zmin$, a Markov process $Z$ such that, once stopped at its first explosion time 
\[\zeta_\infty:=\inf\{t>0: Z_{t-} \text{ or } Z_t=\infty\},\] it has the same law as $(Z^{\mathrm{min}}_t,t\geq 0)$. As noticed in \cite{MR3940763}, c\`adl\`ag extensions of the minimal process may exist with different boundary conditions at $\infty$. 
\\ 

The starting point of the study in \cite{MR3940763} is the following identity for the generator $\mathscr{L}$.

\begin{lemma}[Lemma 5.1 in \cite{MR3940763}]\label{lemmadualityLA} Define $e_x(z):=e^{-xz}=:e_z(x)$ for any $x,z\in (0,\infty)$. One has
\begin{equation}\label{dualityLA}
\mathscr{L}e_x(z)=\mathscr{A}e_z(x), \quad x,z\in (0,\infty), \end{equation}
with $\mathscr{A}$ the operator defined on $C^2(0,\infty)$ as follows: 
\begin{equation}\label{generatorU0}\mathscr{A}g(x):=\frac{c}{2}xg''(x)-\Psi(x)g'(x), \ x\in (0,\infty).
\end{equation}
\end{lemma}

The duality relationship \eqref{dualityLA} lies at the level of generators and actually covers different possibilities for the associated processes depending on the nature of the boundary $\infty$ of the LCSBP and on the boundary $0$ of the diffusion with generator $\mathscr{A}$. The latter is prescribed by the following integral. Let $x_0>0$ be an arbitrary constant and set \[\mathcal{E}:=\int_0^{x_0} \frac{\ddr x}{x}\exp{\left(\frac{2}{c}\int_x^{x_0}\frac{\Psi(u)}{u}\ddr u\right)}.\]

We sum up in the next theorem, the results obtained  in \cite{MR3940763} on the explosion of LCSBPs, the extensions of the minimal process, as well as their behaviors near the boundary $0$ (extinction) when the boundary $\infty$ is non-absorbing.
\begin{theorem}[Theorems 3.1, 3.3, 3.4 and 3.9 in \cite{MR3940763}] \label{backgroundtheorem}\
\bigskip
\begin{itemize}
\item[i)] \underline{Explosion}: The boundary $\infty$ is accessible for $\zmin$ if and only if $\mathcal{E}<\infty$. \\
\item[ii)] \underline{Feller extensions}: There exists a c\`adl\`ag Feller\footnote{Here Feller means that the semigroup maps continuous bounded functions on $[0,\infty]$ into themselves} process $(Z_t,t\geq 0)$ on $[0,\infty]$ with no negative jumps, extending the minimal process $\zmin$, such that for all $x,z\in [0,\infty]$,  $t\geq 0$
\begin{equation}\label{duality1}
\mathbb{E}_z[e^{-xZ_t}]=\mathbb{E}_x[e^{-zU_t}] 
\end{equation}
where $(U_t,t\geq 0)$ is the weak solution to the SDE 
\begin{equation}\label{sdeU}
\ddr U_t=\sqrt{cU_t}\ddr B_t-\Psi(U_t)\ddr t,\ U_0=x
\end{equation}
with $(B_t,t\geq 0)$ a Brownian motion and with boundary conditions at $0$ given in correspondence with that of $Z$ at $\infty$ as in Table \ref{correspondanceUZ}.
\begin{table}[htpb]
\begin{tabular}{|c|c|c|}
\hline
Integral condition & Boundary of $U$ &  Boundary  of $Z$ \\
\hline
$\mathcal{E}=\infty$ & $0$ exit  &  $\infty$  entrance  \\
\hline
$\mathcal{E}<\infty \text{ and } 2\lambda/c<1$ & $0$ regular absorbing &  $\infty$  regular reflecting\\
\hline
$2\lambda/c\geq 1$ & $0$ entrance &  $\infty$  exit\\
\hline
\end{tabular}
\vspace*{4mm}
\caption{Boundaries $\infty$ and $0$ of $Z,U$.}
\label{correspondanceUZ}
\end{table}
\end{itemize}

%\bigskip
\begin{itemize}
\item[iii)] \underline{Extinction}: If $2\lambda/c<1$ (i.e.  $Z$ has the boundary $\infty$ either entrance or regular reflecting), then 
\begin{itemize}
\item $Z$ converges towards $0$ a.s. if and only if $\Psi(z)\geq 0$ for some $z>0$. 
\item $Z$ gets absorbed at $0$ a.s. if and only if $\Psi(z)\geq 0$ for some $z>0$ and $\int^{\infty}\frac{\ddr x}{\Psi(x)}<\infty$.
\end{itemize}

\begin{table}[htpb]
\begin{tabular}{|c|c|c|}
\hline
Integral condition & Boundary of $U$ &  Boundary  of $Z$ \\
\hline
$\int^{\infty}\frac{\ddr x}{\Psi(x)}=\infty$ & $\infty$ natural  &  $0$  natural\\
\hline
$\int^{\infty}\frac{\ddr x}{\Psi(x)}<\infty$ & $\infty$ entrance &  $0$  exit\\
\hline
\end{tabular}
\vspace*{4mm}
\caption{Boundaries $\infty$ and $0$ of $U,Z$.}
\label{correspondanceUZextinc}
\end{table}
\end{itemize}
\vspace*{-4mm}
\end{theorem} 
\noindent The integral conditions for the classification of the boundaries $0$ and $\infty$ of  $U$ displayed in Tables~ \ref{correspondanceUZ} and \ref{correspondanceUZextinc} can be found in \cite[Lemma 5.2]{MR3940763}. It is shown there that Feller's conditions, see Table~ \ref{Fellerconditions}, can indeed be simplified this way. Notice that $\mathcal{E}=\frac{c}{2}M_U(0,x_0]$, where $M_U$ is the speed measure of $U$, see \eqref{speedmeasure} for the general formula. Moreover, Table~\ref{correspondanceUZ} shows that the boundary $0$ of $U$ is regular -- the process $U$ can access and leave $0$ if it is not stopped -- if and only if $\mathcal{E}<\infty$ and $2\lambda/c<1$. Note also that there are no situations in which \(\infty\) (resp. $0$ ) is natural for the LCSBP (resp. for its Laplace dual).
\medskip

The process $Z$ reflected at its boundary $\infty$ was constructed in \cite[Section~7]{MR3940763} as a limit of LCSBPs whose boundaries $\infty$ are all of entrance type. The duality relationship \eqref{duality1} yields actually the \textit{probability} entrance law of the process $Z$ started from $\infty$ and the fact that $\infty$ is regular reflecting when $\mathcal{E}<\infty$ and $\frac{2\lambda}{c}<1$. Indeed since in this case $0$ is regular absorbing for $U$, by letting $z$ go to $\infty$ for fixed $x$, and $x$ go to $0$ for fixed $z$ in \eqref{duality1}, we see that
\[\mathbb{E}_{\infty}[e^{-xZ_t}]=\mathbb{P}_x(U_t=0)>0 \text{ and } \mathbb{P}_z(Z_t<\infty)=\mathbb{E}_{0+}[e^{-zU_t}]=1, \ z\in [0,\infty], t\geq 0.\]

What happens to the process $Z$ past explosion is therefore entirely determined by the law of the first hitting time of $0$ by $U$. A final result from~\cite{MR3940763} that we need to recall is that when the boundary $\infty$ of $Z$ is regular reflecting, i.e.\ when $\mathcal{E}<\infty$ and $2\lambda/c<1$, it is also regular for itself~\cite[Proposition~7.9]{MR3940763}. This implies that the process possesses a non-degenerate local time at $\infty$. However, the construction in~\cite{MR3940763} provides no information either on this local time or on the excursions away from $\infty$. We shall see that the bidual process $V$ (see~\eqref{dualitydiagram}) is particularly useful for a deeper analysis of LCSBPs in this direction. Finally, note that no duality relationship for the minimal LCSBP $(Z_t^{\min}, t \geq 0)$ was established in~\cite{MR3940763} when $\mathcal{E}<\infty$; establishing such a relationship will be one of our main results (Theorem \ref{theorem4}).

\section{Main results}\label{mainresults}
Let $(U_t,t\geq 0)$ be the diffusion solution to \eqref{sdeU} with boundary $0$ either exit, regular absorbing or entrance according to the behavior at  $\infty$ of $Z$.  As explained in the introduction, we will use the following second duality relationship between $U$ and its so-called Siegmund dual process $V$ satisfying:  for any $x,y\in (0,\infty)$  and $t\geq 0$,
\begin{equation}\label{duality2}\mathbb{P}_x(U_t< y)=\mathbb{P}_y(x< V_t),
\end{equation}
We first state a proposition identifying the process $V$ and specify the correspondences between boundaries of the three processes $U$, $V$ and $Z$. This is a direct application of a general statement for diffusions, established in Section \ref{diffusionsec}, see Theorem \ref{CoxRoesler}.
\begin{proposition}\label{propUVZ}
The Siegmund dual of $(U_t,t\geq 0)$ is the diffusion $(V_t,t\geq 0)$ weak solution to the SDE
\begin{equation} \label{SDEV}
\ddr V_t=\sqrt{cV_t}\ddr B_t+\big(c/2+\Psi(V_t)\big)\ddr t,\ V_0=y \in (0,\infty),
\end{equation}
where $(B_t,t\geq 0)$ is a Brownian motion \footnote{We stress that the processes $U$ and $V$ are meant as weak solutions. The driving Brownian motions, all denoted by $B$, are not supposed to be the same in the stochastic equations \eqref{sdeU} and \eqref{SDEV}.} and whose boundary condition at $0$ and $\infty$ are given in correspondence with that of $U$ in the following way: 
\begin{table}[htpb]
\begin{tabular}{|c|c|c|}
\hline
Integral condition & Boundary of $U$ & Boundary of $V$\\
\hline
$\mathcal{E}=\infty$ & $0$ exit & $0$ entrance \\
\hline
$\mathcal{E}<\infty \text{ and } 2\lambda/c<1$ & $0$ regular absorbing & $0$ regular reflecting\\
\hline
$2\lambda/c\geq 1$ & $0$ entrance & $0$ exit \\
\hline
$\int^{\infty}\frac{\ddr x}{\Psi(x)}=\infty$ & $\infty$ natural & $\infty$ natural \\
\hline
$\int^{\infty}\frac{\ddr x}{\Psi(x)}<\infty$ & $\infty$ entrance & $\infty$ exit\\
\hline
\end{tabular}
\vspace*{4mm}
\caption{Boundaries of $U,V$.}
\label{correspondanceUV}
\end{table}
\vspace*{-6mm}
\end{proposition}

Combining the correspondences shown in Tables~\ref{correspondanceUZ} and~\ref{correspondanceUV}, we obtain Table~\ref{correspondanceVZ}, which relates the boundaries of $V$ and $Z$. Notice that the  boundaries $0$ of $V$ and $\infty$ of $Z$ are exchanged but the behaviors of the processes are not anymore.
\begin{table}[h!]
\begin{tabular}{|c|c|c|}
\hline
Integral condition & Boundary of $V$ &  Boundary  of $Z$ \\
\hline
$\mathcal{E}=\infty$ & $0$ entrance &  $\infty$  entrance  \\
\hline
$\mathcal{E}<\infty \text{ and } 2\lambda/c<1$ & $0$ regular reflecting &  $\infty$  regular reflecting\\
\hline
$2\lambda/c\geq 1$ & $0$ exit &  $\infty$  exit\\
\hline
$\int^{\infty}\frac{\ddr x}{\Psi(x)}=\infty$ & $\infty$ natural &  $0$  natural\\
\hline
$\int^{\infty}\frac{\ddr x}{\Psi(x)}<\infty$ & $\infty$ exit &  $0$  exit\\
\hline
\end{tabular}
\vspace*{3mm}
\caption{Boundaries of $V,Z$.}
\label{correspondanceVZ}
\vspace*{-2mm}
\end{table}

Denote by $T_y$ the first hitting time of $y\in [0,\infty]$ of the diffusion $(V_t,t\geq 0)$ and set $\mathscr{G}$ as its generator:
\begin{equation}\label{generatorV1}
\mathscr{G}f(x):=\frac{c}{2}xf''(x)+\left(\frac{c}{2}+\Psi(x)\right)f'(x), \quad f\in C^2,\ x\in (0,\infty).
\end{equation}
Then, from the general theory of one-dimensional diffusions, see e.g. 
Mandel \cite[Chapter V, Section 1]{zbMATH03287297} and Borodin and Salminen  \cite[Chapter II, Section 10]{MR1912205}, the Laplace transform of $T_y$ is expressed, for any $\theta>0$, as
\begin{equation}\label{hittingtimeV}\mathbb{E}_x[e^{-\theta T_y}]=\begin{cases} \frac{h_\theta^+(x)}{h_\theta^+(y)}, &x\leq y\\
\frac{h_\theta^-(x)}{h_\theta^-(y)}, &x\geq y, \end{cases}\end{equation}
where the functions $h_\theta^-$ and $h_\theta^+$ are $C^2(0,\infty)$ and respectively decreasing and increasing solutions to the equation
\begin{equation}\label{eigenvalueG}
\mathscr{G}h(x):=\frac{c}{2}xh''(x)+\left(\frac{c}{2}+\Psi(x)\right)h'(x)=\theta h(x),  \text{ for all } x\in (0,\infty).
\end{equation}
\bigskip

\noindent When $V$ has $0$ as a regular reflecting boundary, $h_\theta^+$ satisfies furthermore the following boundary condition: 
\begin{center}
$\underset{x\rightarrow 0+}{\lim} \frac{(h^+_\theta)'(x)}{s_V(x)}= 0$, 
\end{center}
where $s_V$ is the derivative of the scale function of $V$.
%, given by $$s_V(x):=S_V'(x)=\frac{1}{cx}e^{-\int_{x_0}^{x}\frac{2\Psi(u)}{cu}\ddr u}, \ x\in (0,\infty).$$ 
Note that in our setting $\infty$ is never regular for $V$, see Table \ref{correspondanceUV}, hence no condition is needed for $h_\theta^-$, see e.g. \cite[Chapter II, page 19]{MR1912205} for this fact. 

Last, for any $\theta>0$, the functions $h_\theta^{-}$ and $h_\theta^{+}$ have also the following properties\footnote{\eqref{hittingtimeV0} can be seen by taking $y=0$ (respectively $\infty$) in \eqref{hittingtimeV}} at $0$ and $\infty$,
\begin{align}\label{hittingtimeV0}
\text{if } 0 \text{ (respectively } \infty\text{) is accessible for } V \text{ then }
h_\theta^{-}(0)<\infty \text{ (respectively } h_\theta^{+}(\infty)<\infty\text{)}.  
\end{align}
Furthermore, since in our setting $\infty$ is either a natural boundary or an exit one (and is therefore absorbing in any case; see Table~\ref{correspondanceUV}), one has\footnote{\eqref{hthetainfty0} can be seen by taking $x=\infty$ in \eqref{hittingtimeV}} 
\begin{equation}\label{hthetainfty0}
h_\theta^{-}(\infty)=0.
\end{equation}
We refer again to \cite[Chapter II, Section 10]{MR1912205} for \eqref{hittingtimeV0} and \eqref{hthetainfty0}.
\smallskip

Let $\zeta_0$ be the extinction time of the process $Z$, i.e. $\zeta_0:=\inf\{t>0: Z_t=0\}$ and recall $\zeta_\infty$ the first explosion time. For any $z\in (0,\infty)$, we denote by $\mathbbm{e}_z$ an exponential random variable independent of $V$  with parameter $z$, and by $T_y^{\mathbbm{e}_z}$ the first hitting time of point $y\in [0,\infty]$ by the diffusion $V$ started from $\mathbbm{e}_z$.
\begin{theorem}[Laplace transform of the extinction time of LCSBPs]\label{theorem1}
Assume $\int^{\infty}\frac{\ddr x}{\Psi(x)}<\infty$. For any $0<z<\infty$ and $\theta>0$,  \begin{align}\label{extinctiontimeZ}
&\mathbb{E}_z[e^{-\theta \zeta_0}]=\int_{0}^{\infty}ze^{-zx}\frac{h_{\theta}^+(x)}{h_\theta^{+}(\infty)}\ddr x=\mathbb{E}[e^{-\theta T_\infty^{\mathbbm{e}_z}}].
\end{align}
In particular, if $\infty$ is not absorbing for $Z$ (i.e. if $2\lambda/c<1$) then $\mathbb{E}_\infty[e^{-\theta \zeta_0}]=\mathbb{E}_0[e^{-\theta T_\infty}]>0$. \\

In addition, if $Z$ does not explode (i.e. $\mathcal{E}=\infty$), then for all $z\in (0,\infty]$,
\[\mathbb{E}_z(\zeta_0)<\infty \text{ if and only if } \int_{0}^{x_0}\frac{\ddr x}{x}e^{-Q(x)}\int_{0}^{x}e^{Q(\eta)}\ddr \eta<\infty,\]
with $Q(x):=\int_{1}^{x}\frac{2\Psi(u)}{cu}\ddr u$. In this case, for all $z\in (0,\infty]$
\begin{equation}\label{meanextinction}\mathbb{E}_z(\zeta_0)=\int_{0}^{\infty}\ddr x\frac{2}{cx}e^{-Q(x)}\int_0^x(1-e^{-zv})e^{Q(v)}\ddr v.
\end{equation}

\end{theorem}
\begin{remark} 
The Laplace transform of $\zeta_0$ can be studied via the second order differential equation \eqref{eigenvalueG}. In a more probabilistic fashion, the identity \eqref{extinctiontimeZ} ensures that the time of extinction of the LCSBP started from $z$ has the same law as the time of explosion of the diffusion $V$ started from an independent exponential variable with parameter $z$. The problem of studying $\zeta_0$ is thus transfered into the study of $T_\infty$.
\end{remark}
\begin{remark} Extinction of LCSBPs has been 
studied in \cite{MR2134113} under a log-moment assumption, called $(L)$, on the  Lévy measure $\pi$: $\int^{\infty}\log(h)\pi(\ddr h)<\infty$. Lambert has found, amongst other things, a representation of the Laplace transform of the extinction time in terms of the implicit solution of a certain non-homogeneous Riccati equation, see \cite[Theorem 3.9]{MR2134113}. Note that $\underset{x\rightarrow 0+}{\lim} Q(x)<\infty$ if and only if the assumption $(L)$ holds, see \cite[Proposition 3.13]{MR3940763}. In this case, we can easily check that the condition for $\mathbb{E}_z(\zeta_0)$ to be finite holds. Moreover \eqref{meanextinction} agrees with Equation (9) in \cite[Theorem 3.9]{MR2134113}, where the parameter of competition is $c$ instead of our $c/2$.
\end{remark}
In the next theorem we study $\zeta_\infty$, the first explosion time of the LCSBP. 

\begin{theorem}[Laplace transform of the first explosion time of LCSBPs]\label{theorem2} Assume $\mathcal{E}<\infty$. For all $z\in (0,\infty)$,\begin{align}
&\mathbb{E}_z[e^{-\theta \zeta_\infty}]=\int_{0}^{\infty}ze^{-zx}\frac{h_{\theta}^-(x)}{h_\theta^{-}(0)}\ddr x=\mathbb{E}[e^{-\theta T_0^{\mathbbm{e}_z}}] \label{explosiontimeZ}.
\end{align}
\end{theorem}
One may wonder how Theorem \ref{theorem1} and Theorem \ref{theorem2}  work in the setting of the case without competition $c=0$. This is explained in Section \ref{remarkwithoutcompetition}.
%\smallskip
\begin{remark}
As before we see that $\zeta_\infty$ under $\mathbb{P}_z$, has the same law as the first time of extinction (i.e. of hitting $0$) of $V$ started from an independent exponential random variable with parameter~$z$. 

\begin{itemize}
\item[i)] The condition  $\mathcal{E}<\infty$ turns out to be Feller's test for accessibility of $0$ for $V$ (that simplifies, since $0$ cannot be natural), Table \ref{correspondanceUV}. This yields also a proof for explosion of the LCSBP based on a duality argument.  

\item[ii)] Moreover, since $\underset{x\rightarrow 0}{\lim } h_\theta^{-}(x)=h_\theta^{-}(0)$, we see by \eqref{explosiontimeZ} and Lebesgue's theorem that for all $\theta>0$, $$\underset{z\rightarrow \infty}{\lim}\mathbb{E}_{z}[e^{-\theta \zeta_\infty}]=\underset{z\rightarrow \infty}{\lim} \int_{0}^{\infty}\ddr x \ ze^{-zx}\mathbb{E}_x[e^{-\theta T_0}]=1.$$ We recover here the fact that the boundary $\infty$ for $Z$ is regular for itself.
\end{itemize}
\end{remark}

We establish now a Laplace duality relationship for the minimal process $Z^{\min}$. We focus on the case \(\mathcal{E}<\infty\), since otherwise, by Theorem~\ref{backgroundtheorem}-ii), the minimal process never hits its boundary ~$\infty$. Moreover, when \(2\lambda/c \geq 1\), the boundary \(\infty\) of \(Z\) is an exit, and hence \(Z\) coincides with the minimal process. Therefore, it remains to consider only the case \(\mathcal{E}<\infty\) and \(2\lambda/c<1\). In this setting, the minimal process \((Z_t^{\min},t\geq 0)\) may be interpreted as the logistic CSBP with \(\infty\) regular absorbing, that is, stopped upon hitting \(\infty\).

\begin{theorem}\label{theorem4} Assume $\mathcal{E}<\infty \text{ and }2\lambda/c<1$. For any $x,z\in (0,\infty)$ and $t\geq 0$ \begin{equation}\label{dualityformin}\mathbb{E}_z[e^{-xZ^{\mathrm{min}}_t}]=\mathbb{E}_x[e^{-zU^{\mathrm{r}}_t}],\end{equation}
with $(U_t^{\mathrm{r}},t\geq 0)$ the diffusion solution to \eqref{sdeU} with boundary $0$ regular reflecting. 
In particular, \begin{equation}\label{tailexplosion}
\mathbb{P}_z(\zeta_\infty>t)=\mathbb{E}_{0+}[e^{-zU_t^{\mathrm{r}}}]<1, \ z\in [0,\infty),\ t>0.
\end{equation}
\end{theorem}

\noindent Theorem \ref{theorem4} completes the classification of the boundaries by extending Table \ref{correspondanceUZ} with the correspondences (unaddressed in \cite{MR3940763}) in Table \ref{correspondancemin}. 

\begin{table}[h!]
\begin{center}
\begin{tabular}{|c|c|c|}
\hline
Integral condition & Boundary of $U$ &  Boundary  of $Z$ \\
\hline
$\mathcal{E}<\infty \text{ and } 2\lambda/c<1$ & $0$ regular reflecting &  $\infty$  regular absorbing\\
\hline
\end{tabular}
\vspace*{4mm}
\caption{}
\label{correspondancemin}
\end{center}
\end{table}
We identify now the inverse local time at $\infty$ of the LCSBP $Z$ with boundary $\infty$ regular reflecting.  Denote by $(L^Z_t,t\geq 0)$ the local time at $\infty$ of $Z$ and by $(\tau^Z_x,0\leq x< \xi)$ its right-continuous inverse, namely for any $x\geq 0$, $\tau^Z_x:=\inf\{t\geq 0: L_t^{Z}>x\}$  and $\xi:=L_\infty^{Z}=\inf\{x\geq 0: \tau_x^{Z}=\infty\}\in (0,\infty]$.  One has, see e.g. \cite[Chapter IV, Theorem 4-(iii)]{Bertoin96},
\[\mathcal{I}:=\overline{\{t\geq 0: Z_t=\infty\}}=\overline{\{\tau^Z_x,0\leq x<\xi\}} \text{ a.s.}\]
Moreover the process $(\tau^Z_x,x<\xi)$ is a subordinator with life-time $\xi$, see \cite[Chapter IV, Theorem 8]{Bertoin96}. We denote  by $\kappa_Z$ its Laplace exponent. Note that since $\infty$ is regular reflecting, the subordinator $\tau^Z$ has no drift. Recall also from Proposition \ref{propUVZ} that $0$ is regular reflecting for the bidual process $V$. We call $(L^V_t,t\geq 0)$ its local time at $0$. 
\begin{theorem}\label{theorem3}
Assume $\infty$ regular reflecting ($\mathcal{E}<\infty \text{ and } 2\lambda/c<1$), $(L^Z_t,t\geq 0)$ has the same law as $(L^V_t,t\geq 0)$. Furthermore $\kappa_Z(0)>0$ (and $\mathcal{I}$ is bounded a.s.) if and only if $\Psi$ is positive in a neigbourhood of $\infty$ (equivalently $-\Psi$ is not the Laplace exponent of a subordinator).
\end{theorem}
\begin{remark} 
When $\kappa_Z(0)>0$ (i.e. when $-\Psi$ is not the Laplace exponent of a subordinator), the subordinator $\tau^Z$ has a finite life-time $\xi$, this corresponds to the fact that the process $Z$ makes an infinite excursion away from infinity. According to Theorem~\ref{backgroundtheorem}-iii) -- see also \cite[Lemma~7.7]{MR3940763} -- the process converges towards $0$ a.s. in its infinite excursion (and is absorbed if and only if Grey's condition holds, see Theorem~\ref{backgroundtheorem}-iii)).
\end{remark}
Theoretically, numerous properties of local times of diffusions can thus be applied to the study of $\kappa_Z$ in order for instance to represent the Lévy measure of $\tau_Z$ or its density, see e.g. Borodin and Salminen  \cite[Chapter II, Section 4]{MR1912205}. The latter quantities have no explicit formula when the branching mechanism $\Psi$ is general. However we can identify the packing and Hausdorff dimensions of $\mathcal{I}$.

\begin{theorem}\label{Hausdorff} Assume $\mathcal{E}<\infty \text{ and } \frac{2\lambda}{c}<1$, \[\mathrm{dim}_P(\mathcal{I})=\mathrm{dim}_H(\mathcal{I})=2\lambda/c\in [0,1) \text{ a.s.}\]
\end{theorem}
\begin{remark} The dimension is zero for all branching mechanisms $\Psi$ such that $\Psi(0)=-\lambda=0$. The equality of the packing and Hausdorff dimensions ensures that the Laplace exponent $\kappa_Z$ has the same lower and upper Blumenthal-Getoor's indices, see Bertoin \cite[Page 41]{subordinators}.
\end{remark}
\begin{example}\label{example} 
\begin{enumerate}
\item A first example is given by the case $\Psi\equiv -\lambda$ with $\lambda>0$. The LCSBP $Z$ is degenerated into a process\footnote{This example is in fact a disguised diffusion; since one can interpret the jump to infinity as a killing term, see  \cite[Chapter 2, Section 6]{MR1912205}.} which decays along the deterministic drift $-\frac{c}{2}Z_t^2\ddr t$ when lying in $(0,\infty)$ and jumps from any $z\in (0,\infty)$ to $\infty$ at rate $\lambda z$. According to Theorem~\ref{backgroundtheorem}-ii), if $2\lambda/c\geq 1$ then the boundary $\infty$ of $Z$ is an exit and if $2\lambda/c<1$, it is a regular reflecting boundary. The bidual process $V$ is the solution to the SDE
\[\ddr V_t=\sqrt{cV_t}\ddr B_t+(c/2-\lambda)\ddr t.\]
Therefore, $V$ is a squared Bessel diffusion with dimension $\delta:=c/2-\lambda\geq0$, or equivalently a CSBP with immigration with mechanisms $(\psi,\phi)$ where $\psi(q)=\frac{c}{2}q^2$ and $\phi(q)=\delta q$. According for instance to  Foucart and Uribe Bravo \cite[Proposition 13]{MR3263091}, the inverse local time at $0$ of $V$ is a stable subordinator with index $2\lambda/c$: \begin{center} $\kappa_V(\theta)=\theta^{\frac{2\lambda}{c}}$ for all $\theta\geq 0$.\end{center} By Theorem \ref{theorem3}, the inverse local time of $Z$ at $\infty$ is also stable with the same index, and the Hausdorff dimension of $\mathcal{I}$ is $2\lambda/c\in (0,1)$.
\item A simple example of LCSBP with $\infty$ reflecting which gets extinct almost surely is the LCSBP with $\Psi(x)=-\lambda+(\alpha-1)x^\alpha$ for all $x\geq 0$, with $d>0$, $\alpha\in (1,2]$. In this case the branching part of the process behaves as a critical stable one before the first jump to $\infty$. When $0<2\lambda/c<1$, the process may visit $\infty$ but $\kappa_Z(0)>0$ and the process gets extinct almost-surely in finite time. The bidual process is the diffusion reflected at $0$ (if $0<2\lambda/c<1$) which is the solution to
\[\ddr V_t=\sqrt{cV_t}\ddr B_t+\big(c/2-\lambda+(\alpha-1)V_t^{\alpha}\big)\ddr t.\]
\item Examples of LCSBPs with $\infty$ regular reflecting and  $\lambda=0$ are provided by certain branching mechanisms with slowly varying property at $0$, see \cite[Example 3.14]{MR3940763}. For instance if $\pi_{|(e,\infty)}(\ddr u)=\frac{\alpha}{u(\log u)^2}\ddr u$ and $2\alpha/c<1$ then  the Tauberian and monotone density theorems, Bingham et al. \cite[Theorem 1.7 and 1.7.2]{Bingham87}, give $\Psi(x)\underset{x\rightarrow 0+}{\sim} -\alpha/\log(1/x)$. One has $\mathcal{E}<\infty$ and by Theorem \ref{Hausdorff}, $\mathrm{dim}_H(\mathcal{I})=0$ a.s..
\end{enumerate}
\end{example}
The LCSBP process $Z$ and its bidual $V$ will also have their long-term regimes closely linked. In fact, \eqref{joiningduals} ensures that the existence of  a limiting distribution for $V$ necessarily entails one for $Z$. When $-\Psi$ is the Laplace exponent of a subordinator, the LCSBP can be positive recurrent or null recurrent, see \cite[Theorem 3.7]{MR3940763} for necessary and sufficient conditions. The LCSBP in Example \ref{example}-(1) for instance is null recurrent. We provide more details in the next theorem. 

By It\^o's theory of excursions, since $Z$ and $V$ are Feller processes with boundary $\infty$ and $0$ regular reflecting, their trajectories can be decomposed into excursions out from their boundary $\infty$ and $0$ respectively, see for instance \cite[Chapter 4, Section 4]{Bertoin96}. The process
$(e_t,t\leq L^Z_\infty)$ defined by setting for all $t>0$,
\[e_t=\left(Z_{s+\tau_{t-}^{Z}}, 0\leq s< \tau_{t}^{Z}-\tau_{t-}^{Z}\right) \text{ if }\tau_{t}^{Z}-\tau_{t-}^{Z}>0 \text{ and } e_t=\partial \text{ an isolated point, otherwise,}\]
is a Poisson point process on the set of c\`adl\`ag excursions out from $\infty$, stopped at the first infinite excursion, with for $\sigma$-finite intensity measure the excursion measure, say, $n_Z$. We denote an excursion of $Z$ by $\epsilon: (\epsilon(t),t\leq \zeta)$ with $\zeta$ its length. 
Similarly, the diffusion $V$ with $0$ regular reflecting has an excursion measure $n_V$ on the set of continuous excursions out of $0$. We shall denote an excursion of $V$ by $\omega:(\omega(t),t\leq\ell)$, with $\ell$ its length. Both boundaries $\infty$ and $0$ being regular reflecting, they are also instantaneous (i.e. they are not holding points, see e.g. \cite[Chapter IV, page 104]{Bertoin96}). Since they are moreover regular for themselves, the excursion measures $n_Z$ and $n_V$ are infinite.

The next two results are initiating the study of the excursion measure of $Z$. The first states a relationship between the excursion measures of $Z$ and $V$, the second provides some information about the law of the infimum of an excursion under $n_Z$ for LCSBPs that converge towards $0$ almost surely.

\begin{theorem}\label{theoremdualinsiden} Assume $\infty$ regular reflecting ($\mathcal{E}<\infty \text{ and } 2\lambda/c<1$).  One has the following identity: for all $x\in [0,\infty)$ and $q>0$,
\begin{equation}\label{dualityexcursionmeasure}
n_Z\left(\int_{0}^{\zeta}e^{-qu}e^{-x\epsilon(u)}\ddr u\right)=n_V\left(\int_{0}^{\ell}e^{-qu}\mathbbm{1}_{(x,\infty)}(\omega(u))\ddr u\right).
\end{equation}
Moreover
\begin{equation}\label{LTstationary} n_Z\left(\int_{0}^{\zeta}e^{-x\epsilon(u)}\ddr u\right)=\int_x^{\infty}e^{\int_{1}^{y}\frac{2\Psi(u)}{cu}\ddr u}\ddr y \in (0,\infty], \quad x\in [0,\infty).
\end{equation}
The integral at the right hand side in \eqref{LTstationary} is finite for some $x>0$ if and only if  $-\Psi$ is the Laplace exponent of a subordinator and at least one of the following condition holds
\begin{equation}\label{condposrec}
\underset{u\rightarrow \infty}{\lim}\frac{\Psi(u)}{u}:=-\delta<0, \pi((0,1))=\infty, \bar{\pi}(0)+\lambda>\frac{c}{2}
\end{equation}
where $\delta$ is the drift of $-\Psi$ and $\bar{\pi}(0)$ the total mass of the Lévy measure. 
In this case, the Lévy measure of $\tau_Z$, $n_Z(\zeta\in \ddr h)$, has a finite first moment, which satisfies $$n_Z(\zeta)=\int_0^{\infty}e^{\int_{1}^{y}\frac{2\Psi(u)}{cu}\ddr u}\ddr y<\infty.$$
\end{theorem}
\begin{remark}
The process $Z$ is positive recurrent if and only if $n_Z(\zeta)<\infty$, see the end of \cite[Chapter 2]{subordinators}.  The conditions for $n_Z(\zeta)<\infty$ match therefore with those for positive recurrence found in \cite[Theorem 3.7]{MR3940763}. See also Remark 3.8 in there. Moreover in case of $n_Z(\zeta)<\infty$, if one renormalises \eqref{LTstationary} by $n_Z(\zeta)$, we recover the Laplace transform of the stationary distribution of $Z$.  This is a consequence of a general result representing the stationary distribution through the excursion measure, see Dellacherie et al. \cite[Chapter XIX.46]{Dellacherieetal}.
\end{remark}

\begin{theorem}\label{infimuminexcursion} Assume $\infty$ regular reflecting ($\mathcal{E}<\infty \text{ and } 2\lambda/c<1$) and that $-\Psi$ is not the Laplace exponent of a subordinator. Denote by $I$ the infimum of an excursion of $Z$. For all $a,b\in (0,\infty),$ \[S_Z(a)n_Z\big(I\leq a\big)=S_Z(b)n_Z\big(I\leq b\big)\]
with \begin{equation}\label{eqS_Z}
S_Z(a):=\int_{0}^{\infty}\frac{\ddr x}{x} e^{-ax}e^{-\int_{1}^{x}\frac{2\Psi(u)}{cu}\ddr u}, \ a\in [0,\infty).
\end{equation}

\end{theorem}
\section{A remark on the case without competition}\label{remarkwithoutcompetition}
We comment on Theorems~\ref{theorem1} and~\ref{theorem2} in the case without competition. Recall that $\infty$ in this case is absorbing when accessible. The Laplace transforms of the times of extinction  and explosion are easily derived from the branching property (and its consequences namely Equations~\eqref{cumulant} and~ \eqref{cumulantu}). We explain here the role of Siegmund duality in the case $c=0$. 

Recall that when there is no competition, Equation \eqref{cumulant} states that the Laplace dual process started at $x$ of the CSBP $(Y_t,t\geq 0)$ started at $z$ is the deterministic map $(u_t(x),t\geq 0)$ solution to \eqref{cumulantu}. Moreover, since $0$ and $\infty$ are both absorbing for $Y$, by letting $x$ go to $\infty$ and to $0$ in \eqref{cumulant}, we get
\[\mathbb{P}_z(\zeta^Y_0\leq t)=e^{-zu_t(\infty)} \text{ and } \mathbb{P}_z(\zeta^Y_\infty>t)=e^{-zu_t(0+)},\]
where we denote by $\zeta_0^{Y}$ and $\zeta_\infty^Y$ the extinction and explosion time of $Y$. We assume now that Grey's and Dynkin's conditions \eqref{GreyDynkincond} are satisfied so that both times $\zeta_0^Y$ and $\zeta_\infty^Y$ are finite with positive probability and look for their Laplace transforms.

Let $\rho$ be the largest root of $\Psi$, $\rho:=\sup\{q>0: \Psi(q)\leq 0\}\geq 0$. Note that, for all $q\in [0,\infty]$, $\underset{t\rightarrow \infty}{\lim} u_t(q)= \rho$. By the change of variable $x=u_t(\infty)$, using the fact that $t=\int_{x}^{\infty}\frac{\ddr u}{\Psi(u)}$, and performing an integration by parts, we see  that for any $z\in (0,\infty)$ and $\theta>0$,
\begin{align}
\label{LTextinctionY}\mathbb{E}_z[e^{-\theta \zeta^Y_0}]=\mathbb{P}_z(\zeta^Y	_0\leq \mathbbm{e}_{\theta})&=\int_{0}^{\infty}\theta e^{-\theta t}e^{-zu_t(\infty)}\ddr t= \int_\rho^{\infty}\theta e^{-xz-\theta\int_{x}^{\infty}\frac{\ddr u}{\Psi(u)}}\frac{\ddr x}{\Psi(x)} \nonumber\\
&=\left[e^{-\theta \int_x^{\infty}\frac{\ddr u}{\Psi(u)}}e^{-xz}\right]_{x=\rho}^{x=\infty}+\int_\rho^{\infty}ze^{-xz-\theta\int_{x}^{\infty}\frac{\ddr u}{\Psi(u)}}\ddr x\nonumber\\
&=\int_\rho^{\infty}ze^{-xz-\theta\int_{x}^{\infty}\frac{\ddr u}{\Psi(u)}}\ddr x,
\end{align}
where in the penultimate equality, the bracket term is vanishing at $x=\infty$ and at $x=\rho$. For the limit at $\rho$, note indeed that $\Psi(\rho)=0$, hence $\Psi(u)\underset{u\rightarrow \rho}{\sim} \psi'(\rho)(u-\rho)$ and then $\int_\rho \frac{\ddr u}{\Psi(u)}=+\infty$.

Similarly with $x=u_t(0+)$, using that $t=\int_{0}^{x}\frac{\ddr u}{-\Psi(u)}$, we get
\begin{equation}\label{LTexplosionY}
\mathbb{E}_z[e^{-\theta \zeta^Y_\infty}]=\mathbb{P}_z(\zeta^Y_\infty\leq \mathbbm{e}_{\theta})=\int_{0}^{\infty}\theta e^{-\theta t}(1-e^{-zu_t(0+)})\ddr t=\int_0^\rho ze^{-xz-\theta\int_{0}^{x}\frac{\ddr u}{-\Psi(u)}}\ddr x.
\end{equation} 
To understand the connection with the case including competition, observe that the Siegmund dual of $(u_t(x),t\geq 0)$ is simply its inverse flow, defined by $$v_t(y):=\inf\{z\geq 0: u_{t}(z)>y\},$$ which is solution to the equation $ \frac{\ddr}{\ddr t} v_t(y)=\Psi\big(v_t(y)\big), v_0=y$. When $c=0$, Equation \eqref{eigenvalueG}  reduces to a first-order differential equation exhibiting a singularity at $\rho$ when $\rho\in (0,\infty)$. In this case, explicit expressions for the solutions can be obtained. For any fixed  $x_0\in (\rho,\infty)$, the increasing solution on $(x_0,\infty)$ takes the form $h_\theta^{+}(x)=e^{\theta\int_{x_0}^{x}\frac{\ddr u}{\Psi(u)}}$ for any $x>x_0$. Similarly, the decreasing solution on any interval $(0,x_1)$ with $x_1<\rho$ is given by $h_\theta^{-}(x)=e^{\theta\int_{0}^{x}\frac{\ddr u}{\Psi(u)}}$ for any $x<x_1$. 

By considering the solutions $h_\theta^{-}$ and $h_\theta^{+}$ on their maximal interval, we recover the expressions
\[\mathbb{E}_z[e^{-\theta \zeta^Y_0}]=z\int_\rho^{\infty}e^{-xz}\frac{h_\theta^{+}(x)}{h_\theta^{+}(\infty)}\ddr x \text{ and } \mathbb{E}_z[e^{-\theta \zeta^Y_\infty}]=z\int_0^\rho e^{-xz}\frac{h_\theta^{-}(x)}{h_\theta^{-}(0)}\ddr x.\]
Note that $h_\theta^{+}(\infty)< \infty$ if and only if $\int^{\infty}\frac{\ddr x}{\Psi(x)}<\infty$ (Grey's condition for extinction) and $h_\theta^{-}(0)< \infty$ if and only if $\int_{0}\frac{\ddr x}{-\Psi(x)}<\infty$ (Dynkin's condition for explosion). 

When $\rho=0$ (equivalently $\Psi'(0+)\geq 0$) or $\rho=\infty$ (equivalently $\Psi'(\infty):=\underset{x\rightarrow \infty}{\lim}\frac{\Psi(x)}{x}\leq 0$), one can reinterpret \eqref{LTextinctionY} and \eqref{LTexplosionY} in terms of the identities in law
\[\zeta_0^{Y}\overset{\text{law}}{=}t_{\infty}^{\mathbbm{e}_z}, \text{ if } \rho=0 \text{ and }
\zeta_\infty^{Y}\overset{\text{law}}{=}t_{0}^{\mathbbm{e}_z}\text{ if } \rho=\infty, \]
where \[t_\infty^y=\int_y^{\infty}\frac{\ddr u}{\Psi(u)}=\inf\{t>0: v_t(y)=\infty\} \text{ and } t_0^y=\int_0^y\frac{\ddr u}{-\Psi(u)}=\inf\{t>0: v_t(y)=0\}.\]

\section{Proofs of the main results}\label{proofsec}
Recall the Laplace duality relationships  between $Z$ and $U$: \begin{equation*}
\mathscr{L}e_x(z)\overset{\eqref{dualityLA}}{=}\mathscr{A}e_z(x), \  \forall x,z \in (0,\infty) 
\end{equation*}
and 
\begin{equation*} \mathbb{E}_z[e^{-xZ_t}]\overset{\eqref{duality1}}{=}\mathbb{E}_x[e^{-zU_t}], \ \forall x\in (0,\infty), z\in [0,\infty], t\geq 0.
\end{equation*} 
\subsection{Proofs of Proposition~\ref{propUVZ} and the key identities ~\eqref{joiningduals}, \eqref{joiningdualsatboundary}}\label{section:keyidentities}
We start by identifying in law the process $V$ in Siegmund duality with $U$, namely the process $V$ satisfying
\begin{equation*}
\mathbb{P}_y(x< V_t)\overset{\eqref{duality2}}{=}\mathbb{P}_x(U_t< y), \ x,y\in (0,\infty), t\geq 0.
\end{equation*} 
\textbf{Proof of Proposition~\ref{propUVZ}}. Notice that any branching mechanism $\Psi$ belongs to $C^1(0,\infty)$. This is then a direct application of Theorem \ref{CoxRoesler} with $\frac{1}{2}\sigma^2(x)=\frac{c}{2}x$ and $\mu(x)=-\Psi(x)$. By \eqref{generatorV}, $V$ has generator $$\mathscr{G}f(x)=\frac{c}{2}xf''(x)+\left(\frac{c}{2}+\Psi(x)\right)f'(x).$$ Furthermore,  when $U$  has $0$ regular absorbing, $V$ has  $0$ regular reflecting, see Table \ref{correspondanceSiegmund}. \qed 
\\

The coefficients of the diffusions $U$ and $V$ being smooth on $(0,\infty)$, the laws of $U_t$ and $V_t$ have no atom in $(0,\infty)$ when $t>0$ and \eqref{duality2} holds true with large inequalities.  We shall now exploit the two dualities \eqref{duality1} and \eqref{duality2}. Let $\mathbbm{e}$ be an exponentially distributed random variable with parameter $1$ independent of everything else. For any $q>0$, we denote by $\mathbbm{e}_q:=\mathbbm{e}/q$. 
\\

\noindent \textbf{Proofs of identities \eqref{joiningduals} and \eqref{joiningdualsatboundary}}. We now  link the semigroup of $Z$ to that of $V$. One has by Laplace duality \eqref{duality1} and then Siegmund duality \eqref{duality2},  for any $x,z\in (0,\infty)$, $t\geq 0$:
\begin{equation}\label{firstbidual}
\mathbb{E}_z[e^{-xZ_t}]=\mathbb{E}_x[e^{-zU_t}]=\mathbb{P}_x(\mathbbm{e}_z>U_t)=\int_{0}^{\infty}ze^{-zy}\mathbb{P}_y(V_t>x)\ddr y.\end{equation}

By letting $z$ go to $\infty$, we get the following key identity \eqref{joiningdualsatboundary}
\begin{equation}\label{joiningdualsatboundary2}
\mathbb{E}_\infty(e^{-xZ_t})=\mathbb{P}_0(V_t>x) \text{ for } t,x\geq 0.
\end{equation} 
\qed

We now address the proofs of the main results. Recall the maps $h_\theta^{+}$ and $h_{\theta}^{-}$ and \eqref{hittingtimeV}. 
\subsection{Proof of Theorem \ref{theorem1}}\label{prooftheorem1}
By letting $x$ go to $\infty$, and recalling that $\infty$ is an absorbing boundary for the process $V$, see Table \ref{correspondanceUV}, we get for all $t\geq 0$
\[\mathbb{P}_z(\zeta_0\leq t)=\underset{x\rightarrow\infty}{\lim}\mathbb{E}_z[e^{-xZ_t}]=\int_{0}^{\infty}ze^{-zy}\mathbb{P}_y(V_t=\infty)\ddr y=\int_{0}^{\infty}ze^{-zy}\mathbb{P}_y(T_\infty\leq t)\ddr y.\]
Hence for any $\theta \in (0,\infty)$
\[\mathbb{E}_z[e^{-\theta \zeta_0}]=\mathbb{P}_z(\zeta_0\leq \mathbbm{e}_\theta)=\int_{0}^{\infty}ze^{-zy}\mathbb{P}_y(T_\infty\leq \mathbbm{e}_\theta)\ddr y=\int_{0}^{\infty}ze^{-zy}\mathbb{E}_y\left[e^{-\theta T_\infty}\right]\ddr y.\]
The form in \eqref{extinctiontimeZ} is provided by the identity for diffusions \eqref{hittingtimeV}: $\mathbb{E}_y\left[e^{-\theta T_\infty}\right]=\frac{h_\theta^{+}(y)}{h_\theta^{+}(\infty)}$. We now study $\mathbb{E}_z(\zeta_0)$  under the assumption of non explosion $\mathcal{E}=\infty$. Recall also the assumption $\int^{\infty}\frac{\ddr u}{\Psi(u)}<\infty$, which entails $\zeta_0<\infty$ a.s., see Theorem \ref{backgroundtheorem}-ii). Note that this entails that $0$ is non-attracting for $V$ (from Table \ref{correspondanceUV}, $0$ is actually an entrance) and $\infty$ is an exit for $V$. We need to compute $\mathbb{E}(T_\infty^{\mathbbm{e}_z})$. The calculation is a bit cumbersome but follows from a general result of diffusions, see \cite[Equation (6.6), page 227]{zbMATH03736679}. Recall that $c/2$ is the competition parameter, see \eqref{genLCSBP}. Let $s_V$ be the derivative of the scale function and $m_V$ the speed density measure of $V$ see \eqref{scalefunction} and \eqref{speedmeasure}, with $x_0=1$. One has 
\begin{align}\label{scalefunctionV}
S'_V(y)&=s_V(y)=\frac{1}{y}e^{-\int_{1}^{y}\frac{2\Psi(u)}{cu}\ddr u},
\end{align}
\begin{align}\label{speedmeasureV}
M'_V(y)&=m_V(y)
=\frac{1}{c}e^{\int_{1}^{y}\frac{2\Psi(u)}{cu}\ddr u}.
\end{align}
Then, for any $a>0$,
\begin{equation*}
\mathbb{E}_x(T_a\wedge T_\infty)=2\frac{S_V[a,x]}{S_V[a,\infty]}\int_x^{\infty}S_V(\eta,\infty]\ddr M_V(\eta)+2\frac{S_V[x,\infty]}{S_V[a,\infty]}\int_a^xS_V(a,\eta]\ddr M_V(\eta),
\end{equation*}
and since $S_V(0,x]=\underset{a\rightarrow 0+}{\lim} S_V(a,x]= \infty$ and $S_V[x,\infty]<\infty$,
we see that by letting $a$ go to $0$,  
\[\frac{S_V[a,x]}{S_V[a,\infty]}\underset{a\rightarrow 0}{\longrightarrow} 1 \text{ and } \frac{S_V[x,\infty]}{S_V[a,\infty]}\underset{a\rightarrow 0}{\longrightarrow} 0.\]
Thus
\begin{align*}
\mathbb{E}_\eta(T_\infty)&=2\int_{\eta}^{\infty}S_V(v,\infty]m_V(v)\ddr v
=\int_{\eta}^{\infty}\ddr v\int_v^\infty \frac{2}{cx}e^{-\int_{v}^{x}\frac{2\Psi(u)}{cu}\ddr u}\ddr x.
\end{align*}
Set $Q(x):=\int_{1}^{x}\frac{2\Psi(u)}{cu}\ddr u$ for any $x>0$. We obtain
\begin{align*}
\mathbb{E}_z(\zeta_0)=\int_0^\infty ze^{-z\eta} \mathbb{E}_\eta(T_\infty)\ddr \eta&=\int_0^\infty ze^{-z\eta}\ddr \eta \int_{\eta}^{\infty}\ddr v\int_v^\infty \frac{2}{cx}e^{-\int_{v}^{x}\frac{2\Psi(u)}{cu}\ddr u}\ddr x\\
&=\int_{0}^{\infty}\ddr v(1-e^{-zv})\int_v^{\infty}\frac{2}{cx}e^{-\int_{v}^{x}\frac{2\Psi(u)}{cu}\ddr u}\ddr x\\
&=\int_{0}^{\infty}\ddr x\frac{2}{cx}e^{-Q(x)}\int_0^x(1-e^{-zv})e^{Q(v)}\ddr v, 
\end{align*}
where in the last two equalities we have applied Fubini-Tonelli's theorem. 
\\

We now study the finiteness of $\mathbb{E}_z(\zeta_0)$. First, recall that by assumption $\mathcal{E}=\infty$ and therefore that the process $Z$ has its boundary $\infty$ as entrance, see Table \ref{correspondanceVZ}. According to \cite[Lemma 5.4]{MR3940763}, for any $z\in (0,\infty)$, $\mathbb{E}_\infty(\zeta_z)<\infty$ where we denote by $\zeta_z$ the first passage time below $z$. By the strong Markov property and the fact that there is no negative jumps in $Z$, we see that $\mathbb{E}_\infty(\zeta_0)=\mathbb{E}_\infty(\zeta_z)+\mathbb{E}_z(\zeta_0)$. We can therefore focus on $\mathbb{E}_\infty(\zeta_0)$ in order to see whether $\mathbb{E}_z(\zeta_0)<\infty$ or not. By applying Fubini-Tonelli's theorem, we get
\begin{align}\label{twoparts}
\mathbb{E}_\infty(\zeta_0)&=\int_{0}^{\infty}\ddr \eta\int_\eta^\infty\ddr x\frac{2}{cx}e^{-Q(x)}e^{Q(\eta)}\nonumber\\
&=\int_{0}^{1}\ddr \eta\int_\eta^\infty\ddr x\frac{2}{cx}e^{-Q(x)}e^{Q(\eta)}+\int_{1}^{\infty}\ddr \eta\int_\eta^\infty\ddr x\frac{2}{cx}e^{-Q(x)}e^{Q(\eta)}\nonumber \\
&=:I+J. 
\end{align} 
It is established  in \cite[Equation~(7.4), page 31]{MR3940763}, together with the subsequent calculations therein, that $J\leq \int_{1}^{\infty}\frac{\ddr z}{\Psi(z)}<\infty$.  We thus have $\mathbb{E}_z(\zeta_0)<\infty$ if and only if the first integral $I$ in \eqref{twoparts} is finite. 
Furthermore, 
\begin{align*}I&=\int_{0}^{1}\ddr \eta\int_\eta^{1}\ddr x\frac{2}{cx}e^{-Q(x)}e^{Q(\eta)}+\int_0^{1}\ddr \eta\int_{1}^\infty\ddr x\frac{2}{cx}e^{-Q(x)}e^{Q(\eta)}\\
&=I_1+I_2.
\end{align*}
We argue now that $I_2$ is always finite. Since $J<\infty$, we also have $C:=\int_{1}^\infty\ddr x\frac{2}{cx}e^{-Q(x)}<\infty$ and $I_2=C\int_0^{1}e^{Q(\eta)}\ddr \eta$. The calculation in \cite[Equation (7.6), page 32]{MR3940763}  ensures that $\int_0^{1}e^{Q(\eta)}\ddr \eta<~\infty$. Therefore $\mathbb{E}_z(\zeta_0)<\infty$ if and only if $I_1<\infty$. A last application of Fubini-Tonelli's theorem shows that $I_1<\infty$ is equivalent to $\int_{0}^{1}\frac{\ddr x}{x}e^{-Q(x)}\int_{0}^{x}e^{Q(\eta)}\ddr \eta<\infty.$ \qed
\subsection{Proof of Theorem \ref{theorem2}}
Recall $\mathscr{L}$ the generator of the LCSBP $Z$ given in \eqref{genLCSBP}, $h^-_{\theta}$ in \eqref{hittingtimeV} and the facts that when $\mathcal{E}<\infty$, the boundary $0$ is accessible for the bidual process $V$ and  $h_\theta^{-}(0)<\infty$, see \eqref{hittingtimeV0}. Moreover, we recall that $\infty$ is either natural or exit for $V$ and $h_\theta^{-}(\infty)=0$ for any $\theta>0$, see \eqref{hthetainfty0}.
\begin{lemma}\label{invariantfunctionftheta} 
Let $\theta>0$. Assume $\mathcal{E}<\infty$, then the following function 
\begin{equation}\label{fthetaexplicit}
f_\theta^{+}(z):=
\int_{(0,\infty)}ze^{-xz}h^-_{\theta}(x)\ddr x, \ z\in [0,\infty]
\end{equation}
is a well-defined continuous bounded non-decreasing function satisfying $f_\theta^{+}(0)=0$ and
\begin{center}
$f_\theta^{+}(\infty)=h_\theta^{-}(0)<\infty$. 
\end{center}
Moreover, $f_\theta^+\in C^2(0,\infty)$ and for all $z\in (0,\infty)$,
\begin{equation}\label{harmonic}
\mathscr{L}f_\theta^{+}(z)=\theta f_\theta^{+}(z).
\end{equation}
\end{lemma}

\begin{proof}
By assumption $\mathcal{E}<\infty$, hence $0$ is accessible for $V$, we have $h_\theta^-(0)<\infty$ for any $\theta>0$, see \eqref{hittingtimeV0}, and $f_\theta^{+}$ is well-defined.  We show that $f_\theta^{+}$ takes also the following Bernstein's form:
\begin{equation}\label{FT}
f_\theta^{+}(z)=\int_0^{\infty}(1-e^{-xz})(-h_\theta^{-})'(x)\ddr x.
\end{equation}
By applying Fubini-Tonelli's theorem, we get
\begin{align*}
\int_0^{\infty}(1-e^{-xz})(-h_\theta^{-})'(x)\ddr x=\int_{0}^{\infty}\int_0^{x}ze^{-uz}\ddr u\, (-h_\theta^{-})'(x)\ddr x&=\int_0^{\infty}\left(h_\theta^{-}(u)-h_\theta^{-}(\infty)\right)ze^{-uz}\ddr u\\
&=f_\theta^{+}(z)-h^-_\theta(\infty).
\end{align*}
This provides \eqref{FT} since by \eqref{hthetainfty0}, we have $h^-_\theta(\infty)=0$. One has plainly $f_\theta^{+}(0)=0$ and by letting $z$ go to $\infty$ in \eqref{FT}, we see that $f_\theta^{+}(\infty)=h_\theta^{-}(0)$. The facts that $f_\theta^{+}$ is non-decreasing and $C^2$ follows readily from \eqref{FT}.

We now establish \eqref{harmonic}. Let $z\in (0,\infty)$. Recall $e_x(z)=e_z(x)=e^{-xz}$ for all $x\in (0,\infty)$ and $\mathscr{L}e_x(z)=\mathscr{A}e_z(x)$ for all $x,z\in (0,\infty)$, see Lemma \ref{lemmadualityLA}. Note that $\mathscr{L}1=0=\mathscr{A}1$, therefore $\mathscr{L}(1-e_x)(z)=\mathscr{A}(1-e_z)(x)$ for any $x,z\in (0,\infty)$. For any $\theta>0$ and $z>0$, by Fubini-Lebesgue's theorem
\begin{align}
\mathscr{L}f_\theta^{+}(z)&=\int_{0}^{\infty}\mathscr{L}(1-e_x)(z)(-h^{-}_\theta)'(x)\ddr x=\int_0^{\infty}\mathscr{A}(1-e_z)(x)(-h^-_\theta)'(x)\ddr x \nonumber\\
&=\int_0^{\infty}\frac{c}{2}x(1-e_z)''(x)(-h_\theta^-)'(x)\ddr x-\int_0^{\infty}\Psi(x)(1-e_z)'(x)(-h_\theta^{-})'(x))\ddr x \label{eqtoplug1}.
\end{align}
By integration by parts
\begin{align}
\int_0^{\infty}&\frac{c}{2}x(1-e_z)''(x)(-h_\theta^-)'(x)\ddr x \nonumber\\
&=\left[(1-e_z)'(x)\frac{c}{2}x (-h_\theta^{-})'(x)\right]_{x=0}^{x=\infty}-\int_0^{\infty}(1-e_z)'(x)\frac{c}{2} \left((-h_\theta^{-})'(x)+x(-h_\theta^{-})''(x)\right)\ddr x \nonumber.
\end{align}
We verify now that the bracket terms above vanish. Observe that
\begin{align*}
\left[(1-e_z)'(x)\frac{c}{2}x (-h_\theta^{-})'(x)\right]_{x=0}^{x=\infty}&=
-\underset{x\rightarrow \infty}{\lim} ze^{-zx}\frac{c}{2}x(h_\theta^{-})'(x)+\underset{x\rightarrow 0}{\lim} ze^{-zx}\frac{c}{2}x(h_\theta^{-})'(x)=0.
\end{align*}
Both limits indeed vanish since otherwise $(h_\theta^{-})'$ would not be integrable near $\infty$ nor near $0$, which would contradict $h_\theta^{-}(\infty)>-\infty$ and $h_\theta^{-}(0)<\infty$.  
Hence
\begin{equation}
\int_0^{\infty}\frac{c}{2}x(1-e_z)''(x)(-h_\theta^-)'(x)\ddr x =\int_0^{\infty}(1-e_z)'(x)\frac{c}{2}\left((h_\theta^{-})'(x)+x(h_\theta^{-})''(x)\right)\ddr x,\label{eqtoplug2}
\end{equation}
and going back to the calculation of $\mathscr{L}f_\theta^{+}(z)$, and recalling the generator $\mathscr{G}$ of $V$, see \eqref{generatorV1}, we have
\begin{align*}
\mathscr{L}f_\theta^{+}(z)&=\int_0^{\infty}(1-e_z)'(x)\left(\frac{c}{2} \big((h_\theta^{-})'(x)+x(h_\theta^{-})''(x)\big)+\Psi(x) (h_\theta^{-})'(x)\right)\ddr x\\
&=\int_0^{\infty}ze^{-zx}\mathscr{G}h_\theta^{-}(x)\ddr x=\theta\int_0^{\infty}ze^{-zx}h_\theta^{-}(x)\ddr x=\theta f_\theta^{+}(z).
\end{align*}
\end{proof}
\begin{lemma}\label{lemmaLTexplosion} 
Assume $\mathcal{E}<\infty$ then for any $\theta>0$ and $z\in (0,\infty)$
\begin{equation}\label{LTzeta_inftylem}
\mathbb{E}_z[e^{-\theta \zeta_\infty}]=\frac{f^+_\theta(z)}{f^+_\theta(\infty)}=\mathbb{E}[e^{-\theta T_0^{\mathbbm{e}_z}}],
\end{equation}
with $\mathbbm{e}_z$ an exponential random variable with parameter $z$ independent of $V$ and $T_0^{\mathbbm{e}_z}$ the first hitting time of $0$ of $V$ started from $\mathbbm{e}_z$. 
\end{lemma}
\begin{proof}
Let $f\in \mathcal{D}$ and $(M^{\zmin}_t,t\geq 0)$ be the local martingale \eqref{Mzmin}. By ``integration by parts", see e.g. \cite[Section 3.2]{zbMATH07639718},  we get that for any fixed $\theta>0$, the process
\begin{equation}\label{Mtheta}(M^{\theta,\zmin}_t)_{0\leq t<\zeta_0\wedge \zeta_\infty}:=\left(e^{-\theta t}f(Z^{\min}_{t})-\int_0^{t}e^{-\theta s}\big(\mathscr{L}f(Z^{\min}_s)-\theta f(Z^{\min}_s)\big)\ddr s\right)_{0\leq t<\zeta_0\wedge \zeta_\infty}
\end{equation}
is a local martingale. Choose $f=f_\theta^+$ and let $(S_m)_{m\geq 1}$ be a localizing sequence of stopping times for $M^{\theta,\zmin}$. In particular,  one has $S_m\underset{m\rightarrow \infty}{\rightarrow} \zeta_0\wedge\zeta_\infty$ a.s.. Since by \eqref{harmonic}, $\mathscr{L}f_\theta^{+}=\theta f_\theta^{+}$, we see from \eqref{Mtheta} that the process $\left(e^{-\theta t\wedge S_m}f_\theta^{+}(Z^{\min}_{t\wedge S_m}),t\geq 0\right)$ is a martingale, hence for any $z\in (0,\infty)$ and $t\geq 0$,
\[\mathbb{E}_z[e^{-\theta t\wedge S_m}f_\theta^{+}(Z^{\min}_{t\wedge S_m})]=f_\theta^{+}(z).\] 
Since $f_\theta^{+}$ is bounded, one can apply Lebesgue's theorem and by letting $m$ and $t$ go to $\infty$ in the above equality, we have by continuity of $f_\theta^{+}$, 
\[\mathbb{E}_z[e^{-\theta \zeta_0\wedge \zeta_\infty}f_\theta^{+}(Z^{\min}_{ \zeta_0\wedge \zeta_\infty})]=\mathbb{E}_z[e^{-\theta \zeta_0}f_\theta^{+}(0)\mathbbm{1}_{\{\zeta_0<\zeta_\infty\}}]+
\mathbb{E}_z[e^{-\theta \zeta_\infty}f_\theta^{+}(\infty)\mathbbm{1}_{\{\zeta_\infty<\zeta_0\}}]
=f_\theta^{+}(z),\]
for any $z\in [0,\infty)$. Since $f_\theta^{+}(0)=0$ and $\{\zeta_\infty<\infty\}=\{\zeta_\infty<\zeta_0\}$ a.s. ($0$ is absorbing), we finally obtain
\[\mathbb{E}_z[e^{-\theta \zeta_\infty}]=\frac{f_\theta^{+}(z)}{f_\theta^{+}(\infty)}, \ z\in (0,\infty).\] 
For the representation in term of the first hitting time of $0$ of $V$, notice that by \eqref{hittingtimeV} for any $z>0$, \[\mathbb{E}[e^{-\theta T_0^{\mathbbm{e}_z}}]=\int_{0}^{\infty}ze^{-xz}\mathbb{E}_x[e^{-\theta T_0}]\ddr x=\int_{0}^{\infty}ze^{-xz}\frac{h^{-}_\theta(x)}{h^{-}_\theta(0)}\ddr x=\frac{f_\theta^{+}(z)}{f_\theta^{+}(\infty)}.\]
\end{proof}

\subsection{Proof of Theorem \ref{theorem4}}
The most standard method for establishing this kind of duality result is perhaps to apply Ethier-Kurtz's results, see \cite[Theorem 4.11, page 192]{EthierKurtz}, or to show that $g:x\mapsto \mathbb{E}_z[e^{-xZ_t^{\min}}]$ belongs to the domain of the generator of the diffusion $(U_t^{\mathrm{r}},t\geq 0)$ with the boundary $0$ regular reflecting, see Jansen and Kurt \cite[Proposition 1.2]{duality}. Showing the conditions for applying those results does not seem to be an easy task since the boundary behaviors come into play. We will follow another path.

To emphasize the role of the boundary condition at $0$, we denote by $U^{\mathrm{a}}$, $U^{\mathrm{r}}$ and $V^{\mathrm{a}}$, $V^{\mathrm{r}}$ the solutions to \eqref{sdeU} and \eqref{SDEV} with $0$ either regular absorbing or regular reflecting. We will show \eqref{dualityformin} by introducing the Siegmund dual process of $U^{\mathrm{r}}$. An application of Theorem \ref{CoxRoesler} provides the following relationships : for any $x,y\in (0,\infty)$
\begin{equation}\label{siegmundUrVa}
\mathbb{P}_x(U_t^{\mathrm{a}}<y)=\mathbb{P}_y(V_t^{\mathrm{r}}>x) \text{ and }\mathbb{P}_x(U_t^{\mathrm{r}}<y)=\mathbb{P}_y(V_t^{\mathrm{a}}>x).
\end{equation}
Let $\mathbbm{e}_z$ be an exponential random variable with parameter $z$ independent of $U^{\mathrm{r}}$. Note that \[\mathbb{E}_x[e^{-zU^{\mathrm{r}}_t}]=\mathbb{P}_x(\mathbbm{e}_z>U^{\mathrm{r}}_t)=\int_{0}^{\infty}ze^{-zy}\mathbb{P}_y(V_t^{\mathrm{a}}>x)\ddr y,\]
where $V^{a}$ is the diffusion with generator $\mathscr{G}$ and  $0$ is regular absorbing. 

Assume $\mathcal{E}<\infty$ and $2\lambda/c<1$. Let $(Z_t,t\geq 0)$ be the Feller process extending $(Z_t^{\min},t\geq 0)$ with the boundary $\infty$ regular reflecting, see Theorem \ref{backgroundtheorem}-ii). Recall that $Z$ satisfies the duality relationship \eqref{duality1} with the process $U^{\mathrm{a}}$. We introduce the resolvent of $Z$, $\mathcal{R}^q_Z$ defined on $B_b([0,\infty])$ the space of bounded Borelian functions on $[0,\infty]$. An application of the strong Markov property at time $\zeta_\infty$ yields for any function $f$ vanishing at $\infty$
\begin{align}\label{extendedresolventZ}
\mathcal{R}^q_Zf(z)&:=\mathbb{E}_{z}\left(\int_0^{\infty}e^{-q t}f(Z_t)\ddr t\right)=\mathcal{R}_{\zmin}^qf(z)+\mathbb{E}_{z}\left(\int_{\zeta_\infty}^{\infty}e^{-q t}f(Z_t)\ddr t\right)\nonumber \\
&=\mathcal{R}_{\zmin}^qf(z)+\mathbb{E}_{z}[e^{-q \zeta_\infty}]\mathcal{R}_Z^{q}f(\infty), \ z\in [0,\infty]
\end{align}
where $\mathcal{R}_{\zmin}^qf(z)$ is the resolvent of the minimal process $Z^{\min}$.
Let $e_x(z)=e_z(x)=e^{-xz}$. By the dualities with the auxiliary processes $U^{\mathrm{a}}$ and $V^{\mathrm{r}}$, for the extended process: Let $0<z<\infty$ then
\begin{align}\label{dualresolvent}
\mathcal{R}_Z^qe_x(z)&:=\int_0^{\infty}e^{-q t}\mathbb{E}_{z}\left[e_x(Z_t)\right]\ddr t\nonumber\\
&=\int_0^{\infty}e^{-q t}\mathbb{E}_{x}\left[e_z(U^{\mathrm{a}}_t)\right]\ddr t \qquad \qquad \qquad \qquad\text{ (by the Laplace duality \eqref{duality1})} \nonumber\\
&=\int_0^{\infty}e^{-q t}\mathbb{P}_{x}\left(\mathbbm{e}_z>U^{\mathrm{a}}_t\right)\ddr t \nonumber \\
&=\int_0^{\infty}\!\ddr y\, ze^{-yz}\int_0^{\infty}e^{-q t}\mathbb{P}_{y}\left(V^{\mathrm{r}}_t>x\right)\ddr t \qquad\text{ (by the Siegmund duality \eqref{duality2})}  \nonumber
\\
&=\int_0^{\infty}\!\ddr y\, ze^{-yz}\mathcal{R}^q_{V^{\mathrm{r}}}\mathbbm{1}_{(x,\infty)}(y),
\end{align}
where $(V^{\mathrm{r}}_t,t\geq 0)$ is the Siegmund dual diffusion of $(U^{\mathrm{a}}_t,t\geq 0)$ which is reflected at $0$ and $\mathcal{R}^q_{V^{\mathrm{r}}}$ is its resolvent. Similarly as in \eqref{extendedresolventZ}, for any $f$ such that $f(0)=0$, one has the decomposition
\begin{align}\label{extendedresolventV}
\mathcal{R}^q_{V^{\mathrm{r}}}f(y)&=\mathcal{R}_{V^{\mathrm{a}}}^qf(y)+\mathbb{E}_{y}[e^{-q T_0}]\mathcal{R}^q_{V^{\mathrm{r}}}f(0),\ y\in(0,\infty)
\end{align}
with $\mathcal{R}^q_{V^{\mathrm{a}}}$ the resolvent of the process $(V^{\mathrm{a}}_t,t\geq 0)$ the minimal process with generator $\mathscr{G}$ (i.e. the process absorbed at the boundary $0$). Moreover by Theorem \ref{theorem2}, $\mathbb{E}_z[e^{-q\zeta_\infty}]=\mathbb{E}[e^{-qT_0^{\mathbbm{e}_z}}]$, by \eqref{extendedresolventZ}, \eqref{dualresolvent}, we get for $x,z\in (0,\infty)$:
\begin{align*}
\mathcal{R}_{\zmin}^qe_x(z)&=\int_0^{\infty}\ddr y\,ze^{-yz}\mathcal{R}^q_{V^{\mathrm{r}}}\mathbbm{1}_{(x,\infty)}(y)-\mathbb{E}[e^{-qT_0^{\mathbbm{e}_z}}]\mathcal{R}^q_{V^{\mathrm{r}}}\mathbbm{1}_{(x,\infty)}(0)\\
&=\int_0^{\infty}\ddr y\, ze^{-yz}\left(\mathcal{R}^q_{V^{\mathrm{a}}}\mathbbm{1}_{(x,\infty)}(y)+\mathbb{E}_y[e^{-qT_0}]\mathcal{R}^q_{V^{\mathrm{r}}}\mathbbm{1}_{(x,\infty)}(0)\right)-\mathbb{E}[e^{-qT_0^{\mathbbm{e}_z}}]\mathcal{R}^q_{V^{\mathrm{r}}}\mathbbm{1}_{(x,\infty)}(0)\\
&=\int_0^{\infty}\ddr y\,ze^{-yz}\mathcal{R}^q_{V^{\mathrm{a}}}\mathbbm{1}_{(x,\infty)}(y)\\
&=\int_0^{\infty}\ddr y\,ze^{-yz}\int_0^{\infty} e^{-qt}\mathbb{P}_y(V_t^{\mathrm{a}}>x)\ddr t\\
&=\int_0^{\infty}\ddr y\,ze^{-yz}\int_0^{\infty} e^{-qt}\mathbb{P}_x(y>U_t^{\mathrm{r}})\ddr t \qquad \text{ (by the Siegmund duality \eqref{siegmundUrVa})}\\
&=\mathbb{E}_x\left(\int_0^{\infty}e^{-qt}e^{-zU_t^{\mathrm{r}}}\ddr t\right)=\mathcal{R}^q_{U^{\mathrm{r}}}e_z(x).
\end{align*}
Since the functions $t\mapsto \mathbb{E}_z[e^{-xZ_t^{\min}}]$ and $t\mapsto \mathbb{E}_x[e^{-zU_t^{\mathrm{r}}}]$ share the same Laplace transform, they coincide for almost all $t\geq 0$. Moreover, because the sample paths of $Z^{\min}$ are right-continuous, see Section \ref{lamperticonstruction}, and those of the diffusion $U^r$ are continuous, these functions are respectively right-continuous and continuous. They thus coincide for all $t\geq 0$ and we get the following Laplace duality: for any $x,z\in (0,\infty)$ and $t\geq 0$
\begin{equation}\label{dualminproof}\mathbb{E}_z[e^{-xZ_t^{\min}}]=\mathbb{E}_x[e^{-zU_t^{\mathrm{r}}}].
\end{equation}

\qed

\subsection{Proof of Theorem \ref{theorem3}}
Assume $\mathcal{E}<\infty \ \text{ and } \ 2\lambda/c<1$. Let $Z$ be the LCSBP with $\infty$ regular reflecting and $V$ be its bidual process with $0$ regular reflecting. We study here the local time at $\infty$ of $Z$ with the help of that of $V$ at $0$.

We start by a lemma which provides a relationship between the resolvents of $Z$ and $V$. Recall $\mathcal{R}_Z^{\theta}$ and $\mathcal{R}_V^{\theta}$ the resolvents of  $Z$ and $V$. 
\begin{lemma}\label{lemmaextended-dualityresolvent} Let $g$ be an integrable function on $[0,\infty)$. Set for any $x,z\geq 0$ 
\[G(x):=\int_x^\infty g(u)\ddr u \text{ and }F(z):=\int_0^\infty (1-e^{-uz})g(u)\ddr u,\]
then for any $\theta>0$
\begin{equation}\label{resolventdual}\mathcal{R}_Z^{\theta}F(\infty)=\mathcal{R}_V^{\theta}G(0).
\end{equation}
\end{lemma} 
%\begin{remark}
%The identity \eqref{resolventdual} is closely related to \eqref{dualresolventinfinity}. 
%\end{remark}
\noindent \textit{Proof of Lemma \ref{lemmaextended-dualityresolvent}}. Note first that for any $t\geq 0$, $\mathbb{E}_0[G(V_t)]=\int_0^{\infty}g(u)\mathbb{P}_0(u\geq V_t)\ddr u$. By \eqref{joiningdualsatboundary2}, one has for any $t\geq 0$, $\mathbb{P}_0(u\geq V_t)=\mathbb{E}_\infty(1-e^{-uZ_t})$.
Hence for any $\theta>0$, by Fubini-Lebesgue's theorem
\begin{align*}
\mathcal{R}_V^{\theta}G(0)=\int_0^{\infty}e^{-\theta t}\mathbb{E}_0[G(V_t)]\ddr t&=\int_0^\infty\int_0^\infty e^{-\theta t}g(u)\mathbb{P}_0(u\geq V_t)\ddr u\ddr t\\
&=\int_0^\infty\int_0^\infty e^{-\theta t}g(u)\mathbb{E}_\infty(1-e^{-uZ_t})\ddr u \ddr t\\
&=\int_0^\infty e^{-\theta t}\mathbb{E}_\infty[F(Z_t)]\ddr t=\mathcal{R}_Z^{\theta}F(\infty).
\end{align*}
\qed

Denote by $(L_t^Z,t\geq 0)$ and $(L_t^V,t\geq 0)$ the local times at $\infty$ and $0$ respectively of the processes $Z$ and $V$. We are going to show that they have the same law by establishing that their inverse local times have the same Laplace exponent. 

We will apply some fundamental results due to Blumenthal and Getoor, see \cite{zbMATH03205726}, explaining how the $\theta$-potential operators of the local time of a Hunt process\footnote{The processes $Z$ and $V$ being Feller with compact state space $[0,\infty]$, they belong to this class.} can be associated to a specific family of $\theta$-excessive functions and how one can relate the Laplace exponent of the inverse local time to this family. 

In our setting, dealing first with the process $Z$, we define for any $\theta>0$, within the notation of \cite{zbMATH03205726}, 

\begin{align}\label{PhiPsiZ}
\Phi_Z^{1}:z\mapsto \mathbb{E}_z[e^{-\zeta_\infty}] \text{ and } \Psi_Z^{\theta}:=1-(\theta-1)\mathcal{R}_Z^{\theta}\Phi_Z^{1}, 
\end{align}
By \cite[Theorem 1.2]{zbMATH03205726}, the local time $(L_t^{Z},t\geq 0)$ at $\infty$ of $Z$ satisfies for any $z\geq 0$ and any $\theta>0$ 
\[\mathbb{E}_z\left(\int_0^\infty e^{-\theta t}\ddr L_t^{Z}\right)=\Psi_Z^{\theta}(z).\]
Denote by $\tau^Z$ the inverse of the local time $L^{Z}$, that is $\tau_x^{Z}:=\inf\{t>0: L^{Z}_t>x\}$ for any $x\geq 0$.  The process $(\tau_x^{Z},0\leq x<\xi)$ is a subordinator with life-time $\xi$, and a certain Laplace exponent $\kappa_Z$, that is to say, for all $\theta>0$ and $x\geq 0$, $\kappa_Z$ satisfies $\mathbb{E}_\infty(e^{-\theta \tau_x^{Z}})=e^{-x\kappa_Z(\theta)}$. Recall also that $\tau_x^{Z}=\infty$ for $x\geq \xi$ a.s..  One has 
\begin{align*}
\Psi_Z^{\theta}(\infty)&=\mathbb{E}_\infty\left[\int_{0}^{\infty}e^{-\theta t}\ddr L_t^{Z}\right]=\mathbb{E}_\infty\left[\int_{0}^{\infty}e^{-\theta \tau^Z_x}\ddr x\right]=\frac{1}{\kappa_Z(\theta)},
\end{align*}
where the second equality above is obtained by change of variable, see for instance \cite[Proposition 4.9, Chapter 0]{MR1725357}. Hence for any $\theta>0$, $\kappa_Z(\theta)=1/\Psi_Z^\theta(\infty)$. We refer the reader to \cite[Theorem 2.1]{zbMATH03205726} for more details.

Introduce now the analogue $\theta$-excessive functions for $V$ at the boundary $0$:  for any $\theta>0$, \begin{align}\label{PhiPsiV} \Phi_V^{1}:x\mapsto \mathbb{E}_x[e^{-T_0}] \text{ and }  \Psi_V^{\theta}:=1-(\theta-1)\mathcal{R}_V^{\theta}\Phi_V^{1} \end{align} 
with $T_0$ the first hitting time of $0$ of the diffusion $V$. Denote by $\tau^V$ the inverse of the local time $L^{V}$ of $V$ at $0$ and by $\kappa_V$ its Laplace exponent. We have similarly $\kappa_V(\theta)=1/\Psi_V^\theta(0)$ for any $\theta>0$. It only remains to verify that for any $\theta>0$, \begin{equation}\label{keyidentitylocaltime1}
\Psi_Z^{\theta}(\infty)=\Psi_V^\theta(0).
\end{equation} 
We see plainly from the definitions in \eqref{PhiPsiZ} and \eqref{PhiPsiV} that the identity \eqref{keyidentitylocaltime1} is equivalent to
\begin{equation}\label{keyidentitylocaltime2}
\mathcal{R}^{\theta}_Z\Phi_Z^1(\infty)=\mathcal{R}^{\theta}_V\Phi_V^1(0).
\end{equation} 
By Lemma \ref{lemmaLTexplosion}, Equation \eqref{hittingtimeV} and the definitions of $\Phi_Z^{1}$ and $\Phi_V^{1}$ in \eqref{PhiPsiZ} and \eqref{PhiPsiV},
\[\Phi_Z^{1}(z)=\frac{f_1^+(z)}{f_1^{+}(\infty)} \text{ and } \Phi_V^{1}(x)=\frac{h_1^-(x)}{h_1^{-}(0)}.\]
Moreover by Lemma \ref{invariantfunctionftheta},  
$f^+_1(\infty)=h^-_1(0)$, the identity \eqref{keyidentitylocaltime2} is thus equivalent to
\begin{equation}\label{identityresolvent}
\mathcal{R}_Z^{\theta}f_1^{+}(\infty)=\mathcal{R}_V^{\theta}h_1^{-}(0).
\end{equation}
Recall the expression of $f_\theta^+$ in \eqref{FT} and set for any $u,v,z\geq 0$   \[g(u):=(-h_1^-)'(u),\ 
G(v):=\int_{v}^{\infty}g(u)\ddr u=h_1^{-}(v) \text{ and } F(z):=\int_{0}^{\infty}(1-e^{-uz})g(u)\ddr u=f_1^+(z).\]
By applying Lemma \ref{lemmaextended-dualityresolvent}, we see that \eqref{resolventdual} provides \eqref{identityresolvent}. Finally \eqref{keyidentitylocaltime1} is established and we have shown that the inverse local times of $V$ and $Z$ have the same Laplace exponent, namely $\kappa_Z=\kappa_V$.

We now study the killing term in $\kappa_Z$. Denote by $n_V$ the excursion measure of $V$ away from the point $0$. It is known, see Vallois et al. \cite[Theorem 5-(i)]{zbMATH05503396} and Mallein and Yor \cite[Exercice 13.6]{MalleinYor},  that the supremum $M:=\underset{t\leq \ell}{\sup}\ \omega(t)$ of an excursion $\omega$ of $V$ has  ``law" under the excursion measure given by 
$$n_V\left(M>x\right)=\frac{C}{S_V(x)}  \text{ for any } x>0,$$
where $C\in (0,\infty)$ is some constant
and $S_V$ is the scale function of $V$, vanishing at $0$, namely
\begin{equation}\label{SVvanishingat0}
S_V(x)=\int_0^{x}\frac{\ddr u}{u}e^{-\int_1^u \frac{2\Psi(v)}{cv}\ddr v}, \ x\in [0,\infty).
\end{equation}

The killing term in the inverse local time of $V$ is $\kappa_V(0)=n_V(\ell=\infty)$ where $\{\ell=\infty\} $ is the set of excursions with infinite lifetime, i.e those which do not hit $0$. Necessarily these excursions have transient paths drifting towards $\infty$, otherwise, since $0$ is accessible  from any point in $(0,\infty)$, the infinite excursion of $V$ would eventually hit $0$. Hence, using also that $\kappa_Z=\kappa_V$, we obtain
\begin{equation}\label{killingkappa}
\kappa_Z(0)=n_V(\ell=\infty)=n_V(M=\infty)=\frac{C}{S_V(\infty)},
\end{equation}
with, recalling \eqref{SVvanishingat0},  \[S_V(\infty)=\int_0^\infty  \frac{\ddr x}{x}e^{-\int_{1}^{x}\frac{2\Psi(y)}{cy}\ddr y}.\] 
It remains to see that that the condition $\Psi(x)\geq 0$ for large enough $x$ is necessary and sufficient for $S_V(\infty)<\infty$ (i.e. for $\infty$ to be attracting for $V$). We first show that it is sufficient. Let $x_1>1$ be such that $\Psi(x)\geq \Psi(x_1)\geq 0$ for all $x\geq x_1$. The convexity of $\Psi$ and the fact that $\Psi(0)\leq 0$ ensure that for any $q\in (0,1)$ and $x\in (0,\infty)$, $\Psi(qx)\leq q\Psi(x)$, hence $\Psi(qx)/qx\leq \Psi(x)/x$ and the map $x\mapsto \Psi(x)/x$ is nondecreasing. Therefore, $\Psi(x)/x\geq \Psi(x_1)/x_1\geq 0$ for all $x\geq x_1$. This entails \begin{equation}\label{SZfinite}
\int_{x_1}^\infty  \frac{\ddr x}{x}e^{-\int_{1}^{x}\frac{2\Psi(y)}{cy}\ddr y}\leq C\int_{x_1}^\infty  \frac{\ddr x}{x}e^{-\frac{2\Psi(x_1)}{cx_1}x}<\infty,\end{equation}
with $C=e^{-\int_{1}^{x_1}\frac{2\Psi(y)}{cy}\ddr y}>0$. The integrability near $0$ holds by the assumption $\mathcal{E}<\infty$ and we therefore have $S_V(\infty)<\infty$, namely $\kappa_V(0)>0$. For the necessary part, assume that $-\Psi$ is the Laplace exponent of a subordinator, then $\Psi(x)\leq 0$ for all $x\geq 0$ and plainly for any $x_1\geq 1$
\[\int_{x_1}^\infty  \frac{\ddr x}{x}e^{-\int_{1}^{x}\frac{2\Psi(y)}{cy}\ddr y}\geq \int_{x_1}^\infty  \frac{\ddr x}{x}=\infty,\]
so that $\kappa_Z(0)=0$. This concludes the proof of Theorem \ref{theorem3}.
\qed
\subsection{Proof of Theorem \ref{Hausdorff}}\label{subsecdimI}
According to Theorem \ref{theorem3}, the inverse local time at $\infty$ of $Z$, $(\tau_x^{Z}, 0\leq x<L_\infty^Z)$, has the same law as that of the diffusion $V$ at level $0$, $(\tau_x^{V},0\leq x<L_\infty^V)$. This ensures that the random sets 
$\mathcal{I}=\overline{\{t>0: Z_t=\infty\}}=\overline{\{\tau^Z_x, 0\leq x< L_\infty^Z\}}$ and $\mathcal{Z}=\overline{\{t>0: V_t=0\}}=\overline{\{\tau^V_x, 0\leq x< L_\infty^V\}}$ have the same law and therefore the same Hausdorff and Packing dimensions (which are respectively the lower and upper indices of the Laplace exponent $\kappa_Z$, see Bertoin \cite[Chapter 5]{subordinators} for these notions). We will then be able to apply some general results on diffusions in order to compute these two fractal dimensions of the random set $\mathcal{I}$. 

Recall Feller's construction of the diffusion $V$ reflected at $0$, see Section \ref{sec:fellerconstruction} and Equation \eqref{extended-diff}. We see that the zero set of $V$ coincides with that of the diffusion in natural scale $X$ whose speed density measure is $m_X=1/h$ given by $\eqref{densityspeedmeasure}$ with $S=S_V$. We now turn to the study of the dimension of this set.
 
Recall the expression of the scale function of $V$ in \eqref{SVvanishingat0}. For all $y\in [0,\infty)$
\[h(y)=S_V'(S_V^{-1}(y))^2 S_V^{-1}(y)=\frac{1}{S_V^{-1}(y)}e^{\frac{4}{c}\int_{S_V^{-1}(y)}^{1}\frac{\Psi(u)}{u}\ddr u}.\]
The speed measure of $X$ satisfies  \[M_X(x):=\int_0^x m_X(y)\ddr y=\int_0^{x}S_V^{-1}(|y|)e^{-\frac{4}{c}\int_{S_V^{-1}(|y|)}^{1}\frac{\Psi(u)}{u}\ddr u}\ddr y, \ x\in \mathbb{R}.\]
Moreover for $x\geq 0$, 
\begin{align}\label{defF}
M_X(x)-M_X(-x)=2M_X(x)&=\frac{2}{c}\int_0^{x}S_V^{-1}(y)e^{-\frac{4}{c}\int_{S_V^{-1}(y)}^{1}\frac{\Psi(u)}{u}\ddr u}\ddr y\nonumber\\
&=\frac{2}{c}\int_0^{S_V^{-1}(x)}ze^{-\frac{2}{c}\int_z^{1}\frac{\Psi(u)}{u}\ddr u}S_V'(z)\ddr z \nonumber\\
&=\frac{2}{c}\int_0^{S_V^{-1}(x)}e^{-\frac{2}{c}\int_z^{1}\frac{\Psi(u)}{u}\ddr u}\ddr z.
\end{align}

We are now in the setting  of \cite[Corollary 9.8]{subordinators}  where formulas for the Hausdorff dimension and the packing dimension of the zero-set of $X$ are provided with the help of the speed measure of $X$ (in the notation $M_X=F$, $m_X=f$ of \cite{subordinators}). In our case, this yields, by using the identity~\eqref{defF}:

\begin{align}\label{Hausdorffdimgeneralformula}
\mathrm{dim}_H(\mathcal{Z})&=\sup\left\{\rho\leq 1: \underset{x\rightarrow 0+}{\lim}x^{1-1/\rho}\int_0^{S_V^{-1}(x)}e^{-\frac{2}{c}\int_{z}^{1}\frac{\Psi(u)}{u}\ddr u}\ddr z=\infty\right\} \text{ a.s.}
\end{align}
and
\begin{align}\label{Packingdimgeneralformula}
\mathrm{dim}_P(\mathcal{Z})&=\inf\left\{\rho\leq 1: \underset{x\rightarrow 0+}{\lim}x^{1-1/\rho}\int_0^{S_V^{-1}(x)}e^{-\frac{2}{c}\int_{z}^{1}\frac{\Psi(u)}{u}\ddr u}\ddr z=0\right\} \text{ a.s.}
\end{align}

We now study $x\mapsto \int_0^{S_V^{-1}(x)}e^{-\frac{2}{c}\int_{z}^{1}\frac{\Psi(u)}{u}\ddr u}\ddr z$. Recall $\Psi$ in \eqref{LK} and $\Psi(0)=-\lambda$. Set $\Psi_0$ such that $\Psi(u)=-\lambda+\Psi_0(u)$ for all $u\geq 0$. One has \[e^{-\frac{2}{c}\int_{z}^{1}\frac{\Psi(u)}{u}\ddr u}=z^{-2\lambda/c}e^{-\frac{2}{c}\int_z^{1}\frac{\Psi_0(u)}{u}\ddr u}=:z^{-2\lambda/c}L(z).\]
Note that $\Psi_0(u)\underset{u\rightarrow 0}{\longrightarrow} 0$, so that by Karamata's representation theorem, \cite[Theorem 1.3.1]{Bingham87}, $L$ is a slowly varying function at $0$. 
Moreover, by Karamata's theorem, \cite[Proposition 1.5.8]{Bingham87}, for some $C\in (0,\infty)$ 
\begin{equation}\label{keyequivalent}
\int_0^{S_V^{-1}(x)}e^{-\frac{2}{c}\int_{z}^{1}\frac{\Psi(u)}{u}\ddr u}\ddr z=\int_0^{S_V^{-1}(x)}z^{-2\lambda/c}L(z)\ddr z \underset{x\rightarrow 0+}{\sim}C S_V^{-1}(x)^{-2\lambda/c+1}L(S_V^{-1}(x))
\end{equation}
and by definition of $S_V(x)$:
\[S_V(x)=\int_0^{x}\ddr z \frac{z^{2\lambda/c-1}}{L(z)}\underset{x\rightarrow 0+}{\sim}C\frac{x^{2\lambda/c}}{L(x)}.\]
We now divide the proof in two cases according to $\lambda>0$ or $\lambda=0$. 

Assume first $\lambda>0$, so that $S_V$ is regularly varying at $0$ with index $2\lambda/c$ and so is $S_V^{-1}$ with index $c/2\lambda$, see \cite[Theorem 1.5.12]{Bingham87}, namely there is a function $L'$, slowly varying at $0$ such that
$$S_V^{-1}(x) \underset{x\rightarrow 0+}{\sim} \frac{x^{c/2\lambda}}{L'(x)}.$$
Hence
\begin{equation}\label{lastequiv}S_V^{-1}(x)^{-2\lambda/c+1}L(S_V^{-1}(x))=\left(\frac{x^{c/2\lambda}}{L'(x)}\right)^{-2\lambda/c+1}L\left(\frac{x^{c/2\lambda}}{L'(x)}\right)=:x^{c/2\lambda-1}L''(x).
\end{equation}
Therefore  
\[\underset{x\rightarrow 0+}{\lim}x^{1-1/\rho+c/2\lambda-1}L''(x)=\begin{cases}\infty & \text{ if }\rho<\frac{2\lambda}{c}\\
0 & \text{ if } \rho>\frac{2\lambda}{c}.\end{cases}\]
Finally by combining the asymptotic equivalences \eqref{lastequiv}, \eqref{keyequivalent}, we see from \eqref{Hausdorffdimgeneralformula} and \eqref{Packingdimgeneralformula} that almost surely \[\mathrm{dim}_H(\mathcal{Z})=\mathrm{dim}_P(\mathcal{Z})=\frac{2\lambda}{c}.\]
Assume now $\lambda=0$. The function $S_V$ being increasing and slowly varying, its inverse $S_V^{-1}$ is an increasing rapidly varying function at $0$, see \cite[Theorem 2.4.7]{Bingham87}, i.e for $t>1$, \begin{equation}\label{rapidlyvarying} S_V^{-1}(x)/S_V^{-1}(tx)\underset{x\rightarrow 0+}{\longrightarrow} 0.
\end{equation}
Moreover $S_V^{-1}$ has limit $0$ at $0$ and by \eqref{rapidlyvarying}, for any $\beta\in \mathbb{R}$, $S_V^{-1}(x)x^{\beta}\underset{x\rightarrow 0+}{\longrightarrow} 0$. Equation \eqref{keyequivalent} being valid for $\lambda=0$,  we see that any $\rho>0$ satisfies
\[\underset{x\rightarrow 0+}{\lim}x^{1-1/\rho}S_V^{-1}(x)L(S_V^{-1}(x))=0,\]
hence $\mathrm{dim}_H(\mathcal{Z})=\sup\{\emptyset\}=0$ and $\mathrm{dim}_P(\mathcal{Z})=\inf\{[0,1]\}=0$ almost surely. Joining the two cases, we have  that almost surely
\[\mathrm{dim}_H(\mathcal{Z})=\mathrm{dim}_P(\mathcal{Z})=2\lambda/c.\]
We conclude since the random sets $\mathcal{Z}$ and $\mathcal{I}$ have the same law. \qed
\subsection{Proof of Theorem \ref{theoremdualinsiden}}
We still work under the assumption $\mathcal{E}<\infty \text{ and }2\lambda/c<1$. Recall $\mathcal{R}_Z^q$ the $q$-resolvent of the LCSBP $Z$ with $\infty$ regular reflecting, see \eqref{extendedresolventZ}. Let $q>0$. According for instance to \cite[Equation (7), page 120]{Bertoin96}, the excursion measures are satisfying for any $f,g\in B_b([0,\infty])$ such that $f(\infty)=0$ and $g(0)=0$,
\begin{equation} \label{fromnZtoRZ}
n_Z\left(\int_{0}^{\zeta}e^{-qu}f(\epsilon(u))\ddr u\right)=\kappa_Z(q)\mathcal{R}_Z^qf(\infty)
\end{equation}
and 
\begin{equation} \label{fromnVtoRV}
n_V\left(\int_{0}^{\ell}e^{-qu}g(\omega(u))\ddr u\right)=\kappa_V(q)\mathcal{R}_V^qg(0)
\end{equation}
with $\kappa_Z$ and $\kappa_V$  the Laplace exponents of the inverse local times of $Z$ at $\infty$ and of $V$ at $0$. Theorem \ref{theorem3} ensures that $\kappa_Z=\kappa_V$ and we have seen in Equation \eqref{dualresolventinfinity}, that $\mathcal{R}_Z^qe_x(\infty)=\mathcal{R}_V^q\mathbbm{1}_{(x,\infty)}(0)$. Let $x\in (0,\infty)$ be fixed and pick $f(z)=e^{-xz}$, for all $z\in [0,\infty]$ and $g(v)= \mathbbm{1}_{(x,\infty)}(v)$, we get the first targeted identity
\begin{equation} \label{dualitymeasurecor}
n_Z\left(\int_{0}^{\zeta}e^{-qt}e^{-x\epsilon(t)}\ddr t\right)=n_V\left(\int_{0}^{\ell}e^{-qt}\mathbbm{1}_{(x,\infty)}\big(\omega(t)\big) \ddr t\right), \quad q,x\in (0,\infty). 
\end{equation}
By letting $q$ go to $0$ and monotone convergence, we get the following equality:
\begin{equation} \label{LTexcmeasure}
n_Z\left(\int_{0}^{\zeta}e^{-x\epsilon(t)}\ddr t\right)
=n_V\left(\int_{0}^{\ell}\mathbbm{1}_{(x,\infty)}(\omega(t))\ddr t\right).
\end{equation}
Recall $M_V$ the speed measure of $V$ in \eqref{speedmeasureV} and that for any measurable positive function $f$, the invariant measure $M_V$ satisfies  $\int f \ddr M_V=C n_V\left(\int_{0}^{\ell}f(\omega(t))\ddr t\right)$ for a certain positive constant $C$, see \cite[Chapter XIX.46]{Dellacherieetal}, we see that the left-hand side in \eqref{LTexcmeasure} is 
\[\frac{1}{C}M_V(x,\infty)=\frac{1}{C}\int_x^{\infty}m_V(\ddr v)=\frac{1}{C}\int_x^{\infty}e^{\int_{1}^{v}\frac{2\Psi(u)}{cu}\ddr u}\ddr v.\]
It is clearly infinite when $-\Psi$ is not the Laplace exponent of a subordinator, as in this case $\Psi$ is positive in a neighbourhood of $\infty$. When $-\Psi$ is the Laplace exponent of a subordinator, the following necessary and condition was found in \cite{MR3940763}, see Lemma 5.3-1 and its proof, for $M_V(x,\infty)<\infty$ to be finite. Denote by $\delta$ the drift of $-\Psi$ and set
$$\textbf{(A)} \qquad  \delta=0 \text{ and } \bar{\pi}(0)+\lambda\leq c/2.$$  
\begin{itemize}
\item[i)] If $\textbf{(A)}$ is satisfied then for all $x\geq 0$, $M_V(x,\infty)=\infty$ and $Z$ is null recurrent.
\item[ii)] If $\textbf{(A)}$ is not satisfied then for all $x\geq 0$, $M_V(x,\infty)<\infty$. (Integrability at $0$ of $m_V$ comes from the assumption $\frac{2\lambda}{c}<1$) and $Z$ is positive recurrent.
\end{itemize}
This finishes the proof as $\textbf{(A)}$ is not satisfied as soon as  one of the conditions in \eqref{condposrec} holds, see \cite[Remark 3.8]{MR3940763}.
\qed
\begin{remark} Heuristically, when condition $\textbf{(A)}$ holds, the jumps in the LCSBP have a so small activity that the quadratic drift has enough time to push the path  close to $0$. Once at a low level, the process will take an infinite mean time for exploding. This explains the null recurrence.  
\end{remark}
\subsection{Proof of Theorem \ref{infimuminexcursion}} Our objective is to find the law of the infimum of an excursion. 
%By the Markov property under the excursion measure $n_Z$, for any $a>0$ and $z\in [0,\infty)$, under $n_Z$, conditionally on $\{\epsilon(a)=z,a<\zeta\}$, the shifted process $(\epsilon(t+a),0\leq t\leq \zeta-a)$ has the same law as $(\zmin_t,t<\zeta_\infty)$ started from $z$, see e.g. \cite[page 117]{Bertoin96}. 
We start by finding the law of the infimum $Z^{\min}$ started from an arbitrary $z\in (0,\infty)$. 
%We shall use the following lemma stating a dichotomy for the  longterm behavior of the minimal process when $\Psi(z)\geq 0$ for some $z>0$ (i.e. when $-\Psi$ is not the Laplace exponent of a subordinator), see in particular \cite[Lemma 4.5]{MR3940763} and its proof.
%\begin{lemma}[Lemma 4.3 and Lemma 4.5 in \cite{MR3940763}]\label{explodeorgetextinct} Assume $\mathcal{E}<\infty$ and $\Psi(z)\geq 0$ for some $z>0$, then
%$\mathbb{P}_z(Z^{\mathrm{min}}_t\underset{t\rightarrow \infty}{\longrightarrow} 0 \text{ or }\zeta_\infty<\infty)=1$. Moreover, $\mathbb{P}_z(\zmin_t\underset{t\rightarrow \infty}{\longrightarrow} 0)=\frac{S_Z(z)}{S_Z(0)}>0$, with
%\begin{equation*}
%S_Z(a)=\int_{0}^{\infty}\frac{\ddr x}{x} e^{-ax}e^{-\int_{1}^{x}\frac{2\Psi(u)}{cu}\ddr u}, \ a\in [0,\infty).
%\end{equation*}
%\end{lemma}
\begin{lemma}\label{infimumZmin} For any $a\geq 0$, and $z>a$, 
\begin{equation}\label{infzmin} \mathbb{P}_z(\underset{u\geq 0}{\inf} Z_u^{\min}\leq a)=\frac{S_Z(z)}{S_Z(a)}.
\end{equation}
%Moreover, when $\Psi(z)\geq 0$ for some $z>0$, the events $\{\underset{u\geq 0}{\inf} Z_u^{\min}=0\}$ and $\{Z^{\min}_t\underset{t\rightarrow \infty}{\longrightarrow} 0\}$ coincide almost surely.
\end{lemma}
\begin{proof}
Recall from Section \ref{lamperticonstruction} that $(Z_t^{\min},t\geq 0)$ has the same law as a time-changed transient generalized Ornstein-Uhlenbeck process $(R_t,t\geq 0)$ stopped when exiting $(0,\infty)$. The Laplace transform of the first passage time below $a$ of the process $(R_t,t\geq 0)$, $\sigma_a:=\inf\{t\geq 0: R_t\leq a\}$, is given by
\begin{equation}\label{LaplacetransformhittingOU}\mathbb{E}_{z}[e^{-\mu \sigma_a}]=\frac{g_\mu(z)}{g_\mu(a)},
\end{equation}
with for all $\mu>0$ and $x\in [0,\infty)$, 
$g_\mu(x):=\int_{0}^{\infty} x^{2\mu/c}e^{-zx}\frac{1}{x}e^{-\int_{1}^{x}\frac{2\Psi(y)}{cy}\ddr y}\ddr x$. We refer the reader to Shiga \cite[Theorem 3.1]{MR1061937} and \cite[Equation (4.5) page 13]{MR3940763}. One can recognize at the right of $x^{2\mu/c}e^{-zx}$ in the integrand, the derivative of the scale function of $V$, namely
\[s_V(x)=\frac{1}{x}e^{-\int_{1}^{x}\frac{2\Psi(u)}{cu}\ddr u}.\]
By the time-change construction, see Section \ref{sec:Zmindef}, if one lets $\mu$ go to $0$ in \eqref{LaplacetransformhittingOU}, we get 
\begin{equation*}\label{infzmin} \mathbb{P}_z(\underset{u\geq 0}{\inf} Z_u^{\min}\leq a)=\mathbb{P}_z(\sigma_a<\infty)=\underset{\mu \rightarrow 0}{\lim}\frac{g_\mu(z)}{g_\mu(a)}=\frac{\int_{0}^{\infty}e^{-xz}s_V(x)\ddr x}{\int_{0}^{\infty}e^{-xa}s_V(x)\ddr x}=\frac{S_Z(z)}{S_Z(a)}.\end{equation*}
%By assumption $-\Psi$ is not the Laplace exponent of a subordinator, we have seen in the proof of Theorem \ref{theorem3}, see  \eqref{SZfinite}, that this entails $\int^{\infty}s_V(x)\ddr x<\infty$.  In particular  
%\begin{center}
%$S_V(\infty)=\int_{0}^{\infty}s_V(x)\ddr x=S_Z(0)<\infty$ and $\mathbb{P}_z(\underset{u\geq 0}{\inf} Z_u^{\min}=0)=\frac{S_Z(z)}{S_Z(0)}>0$.
%\end{center} By Lemma \ref{explodeorgetextinct}, $\mathbb{P}_z(\zmin_t\underset{t\rightarrow \infty}{\longrightarrow} 0)=\frac{S_Z(z)}{S_Z(0)}$ for all $z\in [0,\infty)$, so that almost surely 
%\begin{center}
%$\{\zmin_t\underset{t\rightarrow \infty}{\longrightarrow} 0\}=\{\underset{u\geq 0}{\inf} Z_u^{\min}=0\}.$\end{center} 
\end{proof}
\begin{remark} One can also verify more directly that  $\mathbb{P}_z(\zeta^{\mathrm{min},-}_a<\infty)=\frac{S_Z(z)}{S_Z(a)}$, with \begin{equation}\label{zeta_a}\zeta^{\mathrm{min},-}_a:=\inf\{t\geq 0: \zmin_t\leq a\},
\end{equation} by checking that $\mathscr{L}S_Z=0$ in the same way as in the proof of Theorem~\ref{theorem2}. Since $Z$ has no negative jumps and $S_Z(z)\leq S_Z(a)<\infty$ for all $z\geq a$, it follows that $(S_Z(\zmin_{t\wedge \zeta_a^{-}}),t\geq 0)$ is a bounded martingale. This, in turn, yields the desired identity
$\mathbb{P}_z(\zeta^{\mathrm{min},-}_a<\infty)=\frac{S_Z(z)}{S_Z(a)}$.
\end{remark}
The next lemma establishes Theorem \ref{infimuminexcursion}.
\begin{lemma} Assume $\mathcal{E}<\infty$ and $\frac{2\lambda}{c}<1$. Let $I=\inf_{0\leq s<\zeta} \epsilon(s)$. For all $a,b\in (0,\infty),$ \[S_Z(a)n_Z\big(I\leq a\big)=S_Z(b)n_Z\big(I\leq b\big).\]
\end{lemma}
\begin{proof}
Let $0<b<a$. Recall $\zeta^{\mathrm{min},-}_{a}$ in \eqref{zeta_a}. By the strong Markov property under the excursion measure at the stopping time $\zeta_{a}^{\mathrm{min},-}$, \cite[Theorem 3.28, Chapter III, pages 102-103]{zbMATH03205726}, and the absence of negative jumps, one has 
\begin{align*}
n_Z(I\leq b)&=n_Z(I\circ \theta_{\zeta^{\mathrm{min},-}_a}\leq b, \zeta^{\mathrm{min},-}_a<\infty)\\
&=n_Z(I\leq a)\mathbb{P}_a(\zeta^{\mathrm{min},-}_b<\infty)\\
&=n_Z(I\leq a)\frac{S_Z(b)}{S_Z(a)},
\end{align*}
where the last equality holds by Lemma \ref{infimumZmin}.
\end{proof}

\section{One-dimensional diffusions on $[0,\infty]$ and Siegmund duality} \label{diffusionsec}

This section deals with one-dimensional diffusions on $[0,\infty]$. We study their so-called Siegmund duals. The results presented below may have independent interest apart from the study of LCSBPs. 

Siegmund \cite[Theorem 1]{MR0431386} has established that a standard  positive Markov process $U$ whose boundary $\infty$ is either inaccessible (entrance or natural) or absorbing (exit or regular absorbing) admits a dual process $V$ such that for all $t$, $u$, $v$,  $\mathbb{P}_u(U_t<v)=\mathbb{P}_v(V_t>u)$ if and only if $U$ is \textit{stochastically monotone}, that is to say for any $t\geq 0$ and $y\in (0,\infty)$, the function $x\mapsto \mathbb{P}_x(U_t\leq y)$ is nonincreasing. 

We provide below a study of Siegmund duality in the framework of diffusions. Stochastic monotonicity of one-dimensional diffusions is well-known. It can be established for instance through a coupling $(U^x,U^{x'})$ of two diffusions with same coefficients started from $x$ and $x'$  with $x'\geq x$. By continuity of the paths and the strong Markov property, it follows that almost surely $U_t^x=U_t^{x'}$ for any time $t\geq \tau:=\inf\{t>0:U_t^x=U_t^{x'}\}$. In particular, this implies that $\mathbb{P}(U_t^x\leq U_t^{x'})=1$ for all $t\geq 0$ and \[\mathbb{P}(U_t^{x'}\leq z)= \mathbb{P}(U_t^{x'}\leq z, U_t^{x}\leq U_t^{x'})\leq \mathbb{P}(U_t^{x}\leq z), \ \forall z\in (0,\infty).\]

A sketch of proof of the next theorem  was provided by Cox and R\"osler in \cite[Theorem 5]{MR724061}. Their proof relied on scaling limits of birth-death processes.   We provide an alternative proof and complete Cox and R\"osler's theorem by considering also the framework of attracting, natural, exit or entrance boundaries. We refer also the reader to Liggett \cite[Chapter II, Section 3]{MR2108619}, Kolokol'tsov \cite{MR2858558} and Assiotis et al. \cite[Lemma 2.2]{Assiotis2019} for works on Siegmund duality.

\begin{theorem}[Diffusions and Siegmund duality]\label{CoxRoesler}\
Let $\sigma^2\in C^2(0,\infty)$ strictly positive on $(0,\infty)$ and $\mu\in C^1(0,\infty)$. Let $(U_t,t\geq 0)$ be a diffusion on $[0,\infty]$  with generator acting on any $f\in C^2_c(0,\infty)$ by $$\mathscr{A}f(x):=\frac{1}{2}\sigma^{2}(x)f''(x)+\mu(x)f'(x), \text{ for all } x\in (0,\infty),$$
such that $0$ is either inaccessible (entrance or natural) or absorbing (exit or regular absorbing).\\

Then for any $0< u,v<\infty$ and any $t\geq 0$
\begin{equation}\label{dualdiffusion}
\mathbb{P}_u(U_t<v)=\mathbb{P}_v(V_{t}>u),
\end{equation}
with $(V_t, t\geq 0)$ a diffusion on $[0,\infty]$ whose generator is \begin{equation}\label{generatorV}\mathscr{G}f(x):=\frac{1}{2}\sigma^{2}(x)f''(x)+\left(\frac{1}{2}\frac{\ddr }{\ddr x}\sigma^{2}(x)-\mu(x)\right)f'(x), \end{equation}
for any $f\in C^2_c(0,\infty)$ and $x\in (0,\infty)$.\\

Moreover, the following correspondences for boundaries  of $U$ and $V$ hold:

\begin{table}[h!]
\begin{tabular}{|c|c|c|}
\hline
Feller's conditions & Boundary of $U$ & Boundary of $V$\\
\hline
$S_U(0,x]<\infty \text{ and } M_U(0,x]<\infty$ &  $0$ regular  & $0$ regular \\
\hline
$S_U(0,x]=\infty \text{ and } J_U(0)<\infty$ &  $0$ entrance & $0$ exit \\
\hline
$M_U(0,x]=\infty \text{ and } I_U(0)<\infty$ &  $0$ exit & $0$ entrance \\
\hline
$I_U(0)=\infty, \ J_U(0)=\infty$ &  $0$ natural & $0$ natural \\
\hline
\end{tabular}
\vspace*{4mm}
\caption{Boundaries of $U,V$.}
\label{correspondanceSiegmund}
\vspace*{-5mm}
\end{table}
\noindent 
When the boundary $0$ of both $U$ and $V$ is regular, if one is absorbing then necessarily the other is reflecting. Similar correspondences hold for the boundary $\infty$ by replacing everywhere $0$ by $\infty$.
\smallskip

Assume that $0$ is natural or absorbing for $U$ (namely $0$ is either natural, exit or regular absorbing), then the longterm behaviors of $U$ and $V$ are	 also related as follows:

\begin{table}[htpb]
\begin{tabular}{|c|c|c|}
\hline
Condition &  $U$ &  $V$\\
\hline
$S_U(0,\infty)<\infty$ &  $\infty \text{ and }0$ attracting & positive recurrence \\
\hline
\end{tabular}
\vspace*{4mm}
\caption{Longterm behaviors of $U,V$.}
\label{correspondancelimit}
\end{table}
Lastly, when $\infty$ and $0$ are attracting for $U$, the stationary law of $V$ satisfies 
\[\mathbb{P}(V_\infty\leq x)=\mathbb{P}_x(U_t\underset{t\rightarrow \infty}{\longrightarrow} \infty)=\frac{S_U(0,x]}{S_U(0,\infty)} \in (0,1) \text{ for any } x\geq 0.\]
\end{theorem}
\begin{proof}
Theorem \ref{CoxRoesler} is obtained  by combining  Lemmas \ref{lemmaminV}, \ref{lemmatable5} and \ref{lemmacvdist} established below.
\end{proof}

\begin{remark}\label{speed=scale} Combining the two first lines of Table \ref{correspondanceSiegmund}, we see that $0$ is non-absorbing for $U$ (i.e. $J_U(0)<\infty$, and $0$ is regular or an entrance) if and only if $0$ is accessible for $V$ (i.e. $I_V(0)<\infty$, and $0$ regular or exit).
\end{remark}
\begin{remark} Theorem \ref{CoxRoesler} holds more generally for a diffusion $U$ taking values in an interval $[\ell,r]$ by replacing everywhere $0$ and $\infty$ respectively by $\ell$ and $r$. It suffices indeed to consider a $C^2$ bijective function $\varphi$ mapping $[\ell, r]$ to $[0,\infty]$ and to apply the theorem to the diffusion $(\bar{U}_t,t\geq 0):=(\varphi(U_t),t\geq 0)$.
\end{remark}
The proof of Theorem \ref{CoxRoesler} is divided in several lemmas. We start by identifying the generator of the process $(V_t,t\geq 0)$ satisfying \eqref{dualdiffusion} when it evolves in $(0,\infty)$.
\begin{lemma}\label{lemmaminV} Let $T:=\inf\{t>0: V_t\notin (0,\infty)\}$. The process $(V_{t\wedge T},t\geq 0)$ has for generator $\mathscr{G}$ given in \eqref{generatorV}.
\end{lemma}
\begin{proof}
We start by establishing that the process $V$, satisfying the duality relationship \eqref{dualdiffusion}, which states that for all $s\geq 0$, $\mathbb{P}_v(V_s>u)=\mathbb{P}_u(U_s< v)$ for all $u,v\in (0,\infty)$,  is Feller. Namely for any bounded continuous function $f$ on $(0,\infty)$ and $s\geq 0$,
$P^V_sf(w)\underset{w\rightarrow v}{\longrightarrow} P^V_sf(v)$. It suffices to show that for all $u,v\in (0,\infty)$,  $\mathbb{P}_w(V_s>u)\underset{w\rightarrow v}{\longrightarrow}\mathbb{P}_v(V_s>u)$ and $\mathbb{P}_v(V_s=u)=0$. On the one hand, under our assumptions, for any $s>0$, the law of $U_s$ has no atom in $(0,\infty)$, see Section \ref{backgroundspeedscale}. The map $$v\mapsto \mathbb{P}_v(u<V_s)=\mathbb{P}_u(U_s<v),$$ is therefore continuous on $(0,\infty)$. Also, by the strong Feller property of $U$, see e.g. Azencott \cite[Proposition 1.11]{zbMATH03459686}, $u\mapsto \mathbb{P}_u(U_s<v)$ is also continuous, hence for any $u,v\in (0,\infty)$,  $$\mathbb{P}_v(V_s>u)=\underset{\epsilon \rightarrow 0}{\lim}\ \mathbb{P}_v(V_s>u+\epsilon)= \mathbb{P}_v(V_s\geq u),$$ which yields $\mathbb{P}_v(V_s=u)=0$. 

We now show that $V$ has generator $\mathscr{G}$. We will show that $V$ satisfies the martingale problem $\text{(MP)}_V$ associated to $(\mathscr{G},C^2_c(0,\infty))$, see Section \ref{sec:MPdiffusion}, namely:
\[\text{For any }F\in C^2_c(0,\infty), \text{ the process } (M^F_t)_{t\geq 0}:=\left(F(V_t)-\int_{0}^{t}\mathscr{G}F(V_s)\ddr s, t\geq 0\right) \text{ is a martingale}.\]
Our arguments are adapted from those in Bertoin and Le Gall \cite[Theorem 5]{LGB2}. We refer also to \cite[Section 6, page 36]{FoMaMa2019} where the case of branching Feller diffusions is treated.

Let $g\in C^1(0,\infty)$ and $f\in C^\infty_c(0,\infty)$. Set $G(x)=\int_{0}^{x}g(u)\ddr u$ and $F(x)=\int_{x}^{\infty}f(t)\ddr t$. By Fubini-Lebesgue's theorem $$\int_{0}^{\infty}\int_{0}^{\infty}g(u)f(x)\mathbbm{1}_{\{x\geq u\}}\ddr u\ddr x=\int_{0}^{\infty}g(u)F(u)\ddr u =\int_{0}^{\infty}f(x)G(x)\ddr x,$$
and $$\int_{0}^{\infty}f(x)\mathbb{P}_u(V_s< x)\ddr x=\mathbb{E}_u[F(V_s)], \ \int_{0}^{\infty}g(u)\mathbb{P}_x(U_s>u)\ddr u=\mathbb{E}_x[G(U_s)].$$
Recall $\mathbb{P}_u(V_s< x)=\mathbb{P}_x(U_s> u)$. Then, integrating this with respect to $f(x)g(u)\ddr x \ddr u$
provides 
\begin{align*}
\int_{0}^{\infty}\ddr u g(u)\mathbb{E}_u[F(V_s)-F(u)]=\int_{0}^{\infty}\ddr x f(x)\mathbb{E}_x[G(U_s)-G(x)].
\end{align*}
Since $(U_s,s\geq 0)$ has generator $\mathscr{A}$ then \[\mathbb{E}_x[G(U_s)-G(x)]=\int_{0}^{s}\mathscr{A}P^U_tG(x)\ddr t.\]
Hence
\begin{align*}
\int_{0}^{\infty}\ddr x f(x)\mathbb{E}_x[G(U_s)-G(x)]&=\int_{0}^{\infty}\ddr x f(x)\int_{0}^{s}\mathscr{A}P^U_tG(x)\ddr t.
\end{align*}
Since $f$ has a compact support, so does $x\mapsto |f(x)\mathscr{A}P_t^UG(x)|$ and the function $(t,x)\mapsto f(x)\mathscr{A}P_t^UG(x)$
is integrable on $(0,s)\times (0,\infty)$. Therefore, we get by applying Fubini-Lebesgue's theorem

\begin{align*}
\int_{0}^{\infty}\ddr x f(x)\int_{0}^{s}\mathscr{A}P^U_tG(x)\ddr t&=\int_{0}^{s} \ddr t\int_{0}^{\infty}\ddr x f(x)\mathscr{A}P^U_tG(x).
\end{align*}
Set $h(x)=P^U_tG(x)$  and $\phi(x)=f'(x)\frac{1}{2}\sigma^{2}(x)+f(x)\left(\frac{1}{2}\frac{\ddr }{\ddr x}\sigma^{2}(x)-\mu(x)\right)$. Note that $h\in C^2(0,\infty)$ and that under our assumptions on the coefficients $ \mu$ and $\sigma$, $\phi\in C_c^1(0,\infty)$. We now compute $\int_{0}^{\infty}\ddr x f(x)\mathscr{A}P^U_tG(x)$. We get by two integration by parts, 
\begin{align*}
&\int_{0}^{\infty}\ddr x f(x)\mathscr{A}h(x)=\int_{0}^{\infty}\ddr x f(x)\left[\frac{1}{2}\sigma^{2}(x)h''(x)+\mu(x)h'(x)\right]\\
&=\left[f(x)\frac{1}{2}\sigma^{2}(x)h'(x)\right]_{0}^{\infty}-\int_{0}^{\infty}\ddr x \left[f'(x)\frac{1}{2}\sigma^{2}(x)+f(x)\frac{1}{2}\frac{\ddr }{\ddr x}\sigma^{2}(x)\right]h'(x)+ \int_{0}^{\infty}\ddr x f(x)\mu(x)h'(x)\\
&=\left[f(x)\frac{1}{2}\sigma^{2}(x)h'(x)\right]_{0}^{\infty}-\int_{0}^{\infty} \phi(x)h'(x)\ddr x\\
&=\left[f(x)\frac{1}{2}\sigma^{2}(x)h'(x)\right]_{0}^{\infty} -\left[\phi(x)h(x)\right]_{0}^{\infty}+\int_{0}^{\infty} \phi'(x)h(x)\ddr x\\
&=\int_{0}^{\infty} \phi'(x)\mathbb{E}_x\left[\int_{0}^{\infty}\ddr ug(u)\mathbbm{1}_{\{u<U_t\}}\right]\ddr x \text{ (since } f \text{ has a  compact support),} \\
&=\int_{0}^{\infty}\ddr u g(u)\int_{0}^{\infty} \phi'(x)\mathbb{P}_u(V_t<x)\ddr x \text{ (by applying Fubini-Lebesgue's theorem and \eqref{dualdiffusion}),}\\
&=-\int_{0}^{\infty}\ddr u g(u)\mathbb{E}_u[\phi(V_t)]=\int_{0}^{\infty}\ddr u g(u)\mathbb{E}_u[\mathscr{G}F(V_t)],
\end{align*}
where in the two last equalities, we applied Fubini-Lebesgue's theorem and used the identity $\mathscr{G}F(x)=-\phi(x)$ for all $x\in (0,\infty)$.

Therefore, for any $g\in C^1(0,\infty)$, $$\int_{0}^{\infty}\ddr u g(u)\mathbb{E}_u\left[F(V_s)-F(u)-\int_{0}^{s}\mathscr{G}F(V_t)\ddr t\right]=0.$$
It follows that \begin{equation}\label{key} \mathbb{E}_u\left[F(V_s)-F(u)-\int_{0}^{s}\mathscr{G}F(V_t)\ddr t\right]=0 \text{ for almost all } u\in (0,\infty).
\end{equation}
Since $V$ satisfies the Feller property on $(0,\infty)$, the map 
$$u\mapsto \mathbb{E}_u\left[F(V_s)-F(u)-\int_{0}^{s}\mathscr{G}F(V_t)\ddr t\right],$$
is continuous on $(0,\infty)$. Hence, by continuity \eqref{key} extends to all $u\in (0,\infty)$. This entails that for all functions of the form $F(x):=\int_x^{\infty}f(u)\ddr u$ with $f\in C_c^\infty(0,\infty)$, namely for all functions $F$ in $C_c^\infty(0,\infty)$, the process $(M_t^{F},t\geq 0)$ is a martingale. 

The claim that it holds for all $F\in C^2_c(0,\infty)$, and then that the process $V$ satisfies $\text{(MP)}_V$, follows from a density argument to go from $C_c^\infty(0,\infty)$ to $C_c^2(0,\infty)$. Indeed, let $F\in C_c^2(0,\infty)$. There is a compact set $H$ and a sequence $(F_n)$ of functions  in $C^\infty_c(0,\infty)$ vanishing on $H^c$ and such that $(F_n)$ converges uniformly towards $F$ as well as the first and second order derivatives. For every $t\geq 0$, the process $M_s^{F_n}-M_s^{F}$ is bounded on $[0,t]$ by a constant $c_n$ which goes to $0$ as $n\rightarrow \infty$. By passing to the limit in the right-hand side of the inequality
\[\left \lvert \mathbb{E}[M_t^F|\mathcal{F}_s]-M_s^F \right\lvert \leq \left\lvert \mathbb{E}[M_t^F|\mathcal{F}_s]-\mathbb{E}[M_t^{F^n}|\mathcal{F}_s]\right\lvert+\left\lvert M_s^{F^n} -M_s^F\right\lvert,\]
we see that $(M_t^{F},t\geq 0)$ is a martingale. The martingale problem is well-posed for the process stopped when reaching its boundaries, see Section \ref{sec:MPdiffusion}, we therefore have established that $V$, up to hitting its boundaries, is a diffusion with generator $\mathscr{G}$.
\end{proof}

We now explain the correspondences between types of boundaries stated in Table \ref{correspondanceSiegmund}. 
\begin{lemma}[Table \ref{correspondanceSiegmund}]\label{lemmatable5} There exists a constant $c_0\in (0,\infty)$ such that
\begin{center}
$S_U=c_0M_V$ and $M_U=\frac{1}{c_0}S_V$ 
\end{center}
and the correspondences in Table \ref{correspondanceSiegmund} hold. 
\end{lemma}
\begin{proof}
Let $\mu^V$ be the drift term of $V$, i.e. $\mu^V(y)=\frac{1}{2}\frac{\ddr}{\ddr y}\sigma^2(y)-\mu(y)$. Simple calculations provide
$$s_V(v)=\exp\left(-\int_{v_0}^v\frac{\mu^V(y)}{\sigma^2(y)/2}\ddr y\right)=\frac{\sigma^2(v_0)}{\sigma^{2}(v)}\frac{1}{s_U(v)}.$$
and $S_V(x)=\int_{1}^{x}s_V(v)\ddr v=\sigma^2(v_0)M_U(x)$. Similarly, one has $m_V(x)=\frac{1}{\sigma^2(x)s_V(x)}=\frac{s_U(x)}{\sigma^2(v_0)}$ and $M_V(x)=\frac{1}{\sigma^2(v_0)}S_U(x)$. Recall $I_U$ and $J_V$ in \eqref{IJ}, we get
\[I_U(l)=\int_{l}^{x}M_{V}(l,x]\ddr S_V(x)=J_V(l).\]
Hence the scale function and speed measure are exchanged (up to some irrelevant constants) by Siegmund duality, as well as the Feller integral tests $I_U$ and $J_V$. It only remains to justify that if $U$ has its boundary $0$ regular absorbing then $V$ has its boundary $0$ regular reflecting. The proof will be similar for $\infty$ and we omit it. If $0$ is regular absorbing for $U$, then by the duality relationship \eqref{dualdiffusion},
$$\mathbb{P}_{0+}(U_t\geq y)=\underset{x\rightarrow 0+}{\lim}\mathbb{P}_x(U_t\geq y)=\mathbb{P}_{y}(V_t=0)=0,$$
and therefore $0$ is regular reflecting for $V$. Table \ref{correspondanceSiegmund} follows.  
\end{proof} 
We now study the limit behaviors displayed in Table  \ref{correspondancelimit} and the stationary distribution of $V$ when it exists.
\begin{lemma}[Table \ref{correspondancelimit}]\label{lemmacvdist} Assume that the boundaries $0$ and $\infty$ are absorbing for $U$. The diffusion $V$ admits a nondegenerate limiting distribution on $(0,\infty)$ if and only if $0$ and $\infty$ are attracting for $U$. Moreover, the limiting distribution, if it exists, has for cumulative distribution function $x\mapsto \mathbb{P}(V_\infty\leq x):=\frac{S_U(0,x]}{S_U(0,\infty)}$.
\end{lemma}
\begin{proof}
Recall that $0$ (respectively $\infty$) is said to be attracting for $U$ if $U$ converges  towards $0$ (respectively $\infty$) with positive probability, see Section \ref{secFellercond}. Both boundaries $0$ and $\infty$ are attracting if and only if $S_U(0,\infty)<\infty$. Moreover, $V$ is positive recurrent if and only if $M_V(0,\infty)<\infty$, and in this case, the limiting distribution of $V$ is the renormalized speed measure of $V$, i.e. $\mathbb{P}(V_\infty\leq x)=\frac{M_V(0,x]}{M_V(0,\infty)}$, see e.g. \cite[Theorem 54.5, page 303]{zbMATH01515832}. Since $S_U=c_0M_V$ for some constant $c_0>0$, we see  that $S_U(0,\infty)<\infty$ is necessary and sufficient for $V$ to be positive recurrent. Letting $t$ go towards $\infty$ in the duality relationship \eqref{dualdiffusion} provides the expression of $\mathbb{P}(V_\infty\leq x)$ in terms of $S_U$.  This allows us to conclude. 
\end{proof}
We provide below a different argument showing Lemma \ref{lemmacvdist} and establishing the convergence in law of $V$ by Siegmund duality.
\begin{proof} (second proof of Lemma \ref{lemmacvdist}).
Recall that $0$ and $\infty$ are assumed to be either natural or absorbing for $U$,  then $$\mathbb{P}_x(U_t\underset{t\rightarrow \infty}{\longrightarrow} \infty)=1-\mathbb{P}_x(U_t\underset{t\rightarrow \infty}{\longrightarrow} 0)=\frac{S_U(0,x]}{S_U(0,\infty)},$$ 
and since $\mathbb{P}_x(U_t\underset{t\rightarrow \infty}{\longrightarrow} 0 \text{ or } U_t\underset{t\rightarrow \infty}{\longrightarrow} \infty)=1$, we have, by Siegmund duality and Lebesgue's theorem
\[\underset{t\rightarrow \infty}{\lim} \mathbb{P}_y(V_t<x)=\underset{t\rightarrow \infty}{\lim} \mathbb{P}_x(U_t>y)=\mathbb{P}_x(U_t\underset{t\rightarrow \infty}{\longrightarrow} \infty)=\frac{S_U(0,x]}{S_U(0,\infty)}.\]
We have here established the convergence in law of $V$ towards $V_\infty$ by duality. If $S_U(0,\infty)=\infty$, then three cases are possible, see \cite[Proposition 5.22]{karatzas}, either $U$ has no limit as $t$ goes to $\infty$, or $U$ converges almost surely towards $0$ or towards $\infty$. In any of those cases, when $S_U(0,\infty)=\infty$, the process $V$ does not have a nondegenerate limiting distribution on $(0,\infty)$ and the proof of Lemma \ref{lemmacvdist} is complete. 
\end{proof}

We mention that the bidual process arises also naturally in the study of certain conditionings of LCSBPs on never becoming extinct; see Foucart, Rivero, and Winter~\cite{FoRivWi2024}. In addition, a broader class of processes -- CSBPs with collisions -- satisfying a duality diagram of the form~\eqref{dualitydiagram} has been introduced by Foucart and Vidmar~\cite{FoucartVidmar}. The bidual process was used there for classifying their longterm behaviors. 
\\
 
\noindent \textbf{Acknowledgements:} Author's research is partially  supported by LABEX MME-DII (ANR11-LBX-0023-01) and the European Union (ERC, SINGER, 101054787). Views and opinions expressed are however those of the author only and do not necessarily reflect those of the European Union or the European Research Council. Neither the European Union nor the granting authority can be held responsible for them. The author is indebted to Matija Vidmar for many stimulating discussions related to this work and is grateful to the referee for a careful and thorough reading of multiple versions of the manuscript.
\\

\noindent \textbf{Data Availability}: Data sharing not applicable to this article as no datasets were generated or analysed during
the current study.\\ 
\textbf{Declarations of interest}: The author declares that he has no known competing financial interests or personal relationships that could have appeared to influence the work reported in this paper.

\end{document}